\documentclass[11pt,a4paper,reqno]{amsart}
\usepackage{amsmath, amssymb, latexsym, enumerate,  graphicx,tikz, stmaryrd, eufrak,enumitem}

\usepackage{thmtools}

\usepackage[pdftex]{hyperref}
\usepackage{cleveref}
\usepackage{bbm}
\usepackage{cases}
\usepackage{array}
\usepackage{tikz-cd}
\usepackage[all]{xy}

\usepackage[hmarginratio={1:1},vmarginratio={1:1},lmargin=80.0pt,tmargin=90.0pt]{geometry}

\usepackage{tikz}
\tikzset{anchorbase/.style={baseline={([yshift=-0.5ex]current bounding box.center)}}}
\usetikzlibrary{decorations.markings}
\usetikzlibrary{decorations.pathreplacing}
\usetikzlibrary{arrows,shapes,positioning,backgrounds}
\tikzstyle directed=[postaction={decorate,decoration={markings,
    mark=at position #1 with {\arrow{>}}}}]
\tikzstyle rdirected=[postaction={decorate,decoration={markings,
    mark=at position #1 with {\arrow{<}}}}]
    
\usepackage{graphicx}
\usepackage{nicefrac}

 \newlength{\baseunit}               
 \newcount{\numlines}                
\newtheorem{theorem}[subsubsection]{Theorem}
\newtheorem{lemma}[theorem]{Lemma}
\newtheorem{prop}[theorem]{Proposition}
\newtheorem{corollary}[subsubsection]{Corollary}
\newtheorem{conjecture}[theorem]{Conjecture}

\theoremstyle{definition}
\newtheorem{definition}[subsubsection]{Definition}

\newtheorem{remark}[theorem]{Remark}

\newtheorem{example}[subsubsection]{Example}

\newtheorem{question}[theorem]{Question}

\newcommand{\fk}{\mathbf{k}}

\newcommand{\bk}{\mathbf{k}}
\newcommand{\JW}{\mathrm{JW}}
\newcommand{\cJW}{\mathcal{J}\mathcal{W}}
\newcommand{\cJWR}{\mathcal{J}\mathcal{W}\mathcal{R}}
\newcommand{\cY}{\mathcal{Y}}
\newcommand{\cC}{\mathcal{C}}

\numberwithin{equation}{section}

\newcommand{\cK}{\mathcal{K}}

\newcommand{\ev}{\mathrm{ev}}
\newcommand{\co}{\mathrm{co}}

\newcommand{\bD}{\mathbf{D}}

\newcommand{\Dim}{\mathrm{Dim}}

\newcommand{\TL}{\mathrm{TL}}

\newcommand{\gd}{\mathsf{gd}}

\newcommand{\FPdim}{\mathrm{FPdim}}

\newcommand{\cD}{\mathcal{D}}

\newcommand{\Ver}{\mathsf{Ver}}

\newcommand{\unit}{{\mathbbm{1}}}
\newcommand{\tto}{\twoheadrightarrow}

\newcommand{\mN}{\mathbb{N}}
\newcommand{\mZ}{\mathbb{Z}}

\newcommand{\mC}{\mathbb{C}}
\newcommand{\mR}{\mathbb{R}}
\newcommand{\mQ}{\mathbb{Q}}

\newcommand{\End}{\mathrm{End}}

\newcommand{\Hom}{\mathrm{Hom}}

\newcommand{\Sym}{\mathrm{Sym}}
\newcommand{\id}{\mathrm{id}}

\newcommand{\Ind}{\mathrm{Ind}}

\newcommand{\Vecc}{\mathsf{Vec}}

\newcommand{\sVec}{\mathsf{sVec}}

\newcommand{\Tilt}{\mathsf{Tilt}}

\newcommand{\mE}{\mathbb{E}}
\newcommand{\mV}{\mathbb{V}}
\newcommand{\mD}{\mathbb{D}}
\newcommand{\Ch}{\mathrm{Ch}}
\newcommand{\cA}{\mathcal{A}}

\newcommand{\Inj}{\mathrm{Inj}}

\newcommand{\bR}{\mathbf{R}}

\newcommand{\mF}{\mathbb{F}}

\newcommand{\gr}{\mathrm{gr}}
\newcommand{\Fr}{\mathrm{Fr}}
\newcommand{\Cone}{\mathrm{Cone}}

\begin{document}
\title{$N$-spherical functors and tensor categories}

\author{Kevin Coulembier}
\address{School of Mathematics and Statistics, University of Sydney, Australia}
\email{kevin.coulembier@sydney.edu.au}
\author{Pavel Etingof}
\address{Department of Mathematics, MIT, Cambridge, MA USA 02139
}\email{etingof@math.mit.edu}

\date{\today}

\keywords{tensor categories, $N$-spherical functors, bounded objects, Verlinde categories, Jones-Wenzl idempotents, Frobenius-Perron dimension}

\begin{abstract}
We apply the recently introduced notion, due to Dyckerhoff, Kapranov and Schechtman, of $N$-spherical functors of stable infinity categories, which generalise spherical functors, to the setting of monoidal categories. We call an object $N$-bounded if the corresponding regular endofunctor on the derived category is $N$-spherical. Besides giving new examples of $N$-spherical functors, the notion of $N$-bounded objects gives surprising connections with Jones-Wenzl idempotents, Frobenius-Perron dimensions and central conjectures in the field of symmetric tensor categories in positive characteristic.
\end{abstract}

\maketitle

\section*{Introduction}

In \cite{DKS}, Dyckerhoff, Kapranov and Schechtman introduced the notion of an `$N$-spherical' functor between stable infinity categories, for $N$ a positive integer. This is a generalisation of the more standard case $N=4$ of a `spherical' functor between triangulated categories. In particular, the definition generalises Kuznetsov's version of the definition of spherical functors in \cite{Ku}. It follows immediately that 3-spherical functors are precisely equivalences, 2-spherical functors are the zero functors and no functors are 1-spherical. The notion of $N$-spherical functors for $N>4$ is fundamentally new.

In the current paper, we apply the definition and theory of \cite{DKS} to the derived category of a tensor category, see \cite{Del02, EGNO}. Concretely, we call an object $X$ of a tensor category {\em $N$-bounded} if the endofunctor $-\otimes X$ of the derived category is $N$-spherical in the sense of \cite{DKS}. This leads to surprisingly strong results and interesting connections with recent developments in the theory of symmetric tensor categories in \cite{BE1, BEO, AbEnv, CEO1, Os}. Moreover, we show that many objects in known tensor categories are $N$-bounded, providing a large number of new examples of $N$-spherical functors. Contrary to most known examples of $N$-spherical functors, which are known to exist indirectly as the gluing functor associated to some periodic semi-orthogonal decomposition, it is the functors themselves which appear naturally here.

For a variety of reasons, the paper is written in a self-contained way, in the sense that we do not rely on the results in \cite{DKS}, but derive directly the required special cases (sometimes with near identical proof). However, the inspiration for the current paper originates entirely from~\cite{DKS}.

Now we describe the paper in more detail.
In Section~\ref{SecPrel} we introduce some required background on tensor categories and Temperley-Lieb algebras. 

Let $\cA$ be an additive monoidal category. In Section~\ref{SecAdd}, we introduce the continuant complexes $\mE_n(X)$, which are chain complexes in $\cA$ associated to every natural number $n$ and every object $X$ in $\cA$ admitting all duals. They categorify Euler's continuant polynomials and are constructed from the evaluations of $X$ and its duals. For example $\mE_2(X)$ is simply the complex
$$0\to X^\ast\otimes X\xrightarrow{\ev} \unit\to 0.$$ 
We also have dual complexes $\mE^n(X)$, constructed from coevaluation. These complexes are based on the theory of \cite{DKS}, so that the endofunctor $-\otimes X$ of the homotopy category~$\cK(\cA)$ is $N$-spherical if and only if $\mE_{N-1}(X)$ and $\mE^{N-1}(X)$ vanish in $\cK(\cA)$. We refer to such objects~$X$ as {\em homotopically $N$-bounded}. Also in Section~\ref{SecAdd} we derive several exact triangles in $\cK(\cA)$ based on the continuant complexes, mostly following \cite{DKS}. These form the backbone for the later applications to tensor categories.

In Section~\ref{SecPiv} we focus on~$X\in\cA$ for which the second left dual $X^{\ast\ast}$ is isomorphic to~$X$, {\it i.e.} the pivotal case. Our main result here is that the homotopical $N$-boundedness of such $X$ is largely governed by the Jones-Wenzl idempotents. Concretely, if the appropriate Jones-Wenzl idempotent exists, then $\mE_n(X)$ is isomorphic in $\cK(\cA)$ to a complex contained in one homological degree, given by the evaluation of the idempotent on some iterated product of~$X^\ast$ and~$X$. In particular, homotopical $N$-boundedness can then be described within $\cA$, rather than its homotopy category. Also the notion of `rotatability' of Jones-Wenzl idempotents, see \cite{El, Ha}, arises naturally here.

In Section~\ref{SecTens1}, we zoom in on tensor categories, meaning rigid abelian $\bk$-linear monoidal categories $\cC$ for which the endomorphism algebra of the tensor unit is the base field $\bk$. In this case we can focus on the weaker assumption that $-\otimes X$ be $N$-spherical on the derived category $\cD(\cC)$ (in the sense of \cite{DKS}), rather than the homotopy category. We refer to such objects as {\em $N$-bounded} objects, and the explicit condition now becomes that $\mE_{N-1}(X)$ and $\mE^{N-1}(X)$ must be acyclic, or in other words, vanish in $\cD(\cC)$. We show that, remarkably, this condition is equivalent with the image of $\mE_{N-1}(X)$, or equivalently of $\mE^{N-1}(X)$, being zero in the Grothendieck ring~$K_0(\cC)$. Henceforth we say that $X$ is `bounded' when it is $N$-bounded for some $N$ and `unbounded' otherwise. We show that bounded objects must be simple. As it turns out, examples of $N$-bounded objects, for all $N>1$, can be found in well-known tensor categories, even when restricting to fusion categories.

In Section~\ref{SecTens2} we focus on various flavours of tensor categories. We show that in finite tensor categories, the homological property of $N$-boundedness is governed entirely by Frobenius-Perron dimension. Concretely, if the Frobenius-Perron dimension of an object is $2\cos(\pi/N)$, then~$X$ is $M$-bounded if and only if $N$ divides $M$. An object is unbounded if and only if its Frobenius-Perron dimension is 2 or higher. We also focus on symmetric tensor categories. We show that the only bounded objects in symmetric tensor categories in characteristic zero are zero or invertible. In sharp contrast, there are interesting examples of bounded objects when the characteristic of $\bk$ is $p>0$. In \cite{BE1, BEO, AbEnv}, new `incompressible' symmetric tensor categories $\Ver_{p^n}$ were introduced and it was conjectured in \cite{BEO} that these play the same role in positive characteristic as the category of supervector spaces plays in characteristic zero in Deligne's results \cite{Del02}. Thus far, only the `$n=1$ case' of this conjecture is proven, see \cite{CEO1}. We show that the generator of $\Ver_{p^n}$ is $p^n$-bounded. Moreover, for $p>2$, there is a bijection between tensor subcategories of $\Ver_{p^n}$ and bounded objects. The aforementioned \cite[Conjecture~1.4]{BEO} implies a complete classification of $N$-bounded objects in symmetric tensor categories in positive characteristic. We prove the validity of this classification for finite tensor categories and for small $N$, hence providing new evidence towards \cite[Conjecture~1.4]{BEO}. We also consider pivotal tensor categories and non-pivotal examples of bounded objects.

Bounded objects in symmetric tensor categories yield examples of objects for which both symmetric powers and exterior powers vanish in high enough degrees. In Section~\ref{SecFiniteP} we focus more on the latter property.

\section{Preliminaries}\label{SecPrel}

\subsection{Notation and conventions}

\subsubsection{} Set $\mN=\{0,1,2,\cdots\}$. We refer to an idempotent complete, additive category as a pseudo-abelian category. If $X$ is an object in an abelian category, we denote its length (the supremum of the lengths of finite strictly ascending filtrations) by $\ell(X)\in\mN\cup\{\infty\}$.

\subsubsection{}

For an additive category $\cA$ we consider the category of chain complexes $\Ch(\cA)$, its homotopy category $\cK(\cA)$ and, in case $\cA$ is abelian, its derived category $\cD(\cA)$. We will use homological conventions and notation for chain complexes. For a chain complex $(C_\bullet, d_\bullet)$, abbreviated to $C$, its shift $C[1]$ is the complex with $C[1]_i=C_{i-1}$ and $- d_i: C[1]_{i+1}\to C[1]_{i}$.

For a chain map $f$ from $(C_\bullet,d^C_\bullet)$ to $(D_\bullet, d_\bullet^D)$, the mapping cone $\Cone(f)$ is
$$\Cone(f)_i=C_{i-1}\oplus D_i $$
with differential
$$
\left(\begin{array}{cc}
-d_{i-1}^C& 0\\
- f_{i-1}& d_i^D
\end{array}\right)\; : \; C_{i-1}\oplus D_i\to C_{i-2}\oplus D_{i-1}.
$$

We denote the canonical fully faithful embedding $\cA\to\Ch(\cA)$ by $X\mapsto X[0]$.

\subsubsection{}\label{duals}We refer to \cite{EGNO} for the basics on monoidal categories.  Let $(\cA,\otimes,\unit)$ be a monoidal category. 

We call an object $X\in\cA$ {\bf rigid} if it admits `all duals'. By this, we mean that $X$ admits a left dual $(X^\ast,\ev_X,\co_X)$, which in turn admits a left dual etc., and similar for right duals. We write $X^{(i)}$ for the $i$-th left adjoint, if $i\in\mN$, or the $-i$-th right adjoint if $i\in\mZ_{<0}$.

If $X\in\cA$ has a left and right adjoint, it leads to adjoint functors on $\cA$
$$-\otimes {}^\ast X \;\,\dashv \,\;-\otimes X\,\;\dashv\;\, -\otimes X^\ast.$$

\subsubsection{} We will only consider {\em additive} monoidal categories $\cA$, where we always assume that the tensor product is additive in each variable. The categories $\Ch(\cA)$, $\cK(\cA)$ and $\cD(\cA)$ inherit monoidal structures, and for rigid objects the adjoint endofunctors on $\cA$ from \ref{duals} extend to these categories.

\subsection{Tensor categories} 
Let $\bk$ be a field.

\subsubsection{}Following \cite{EGNO}, we call an essentially small $\bk$-linear monoidal category $(\cC,\otimes,\unit)$ a {\bf tensor category over $\bk$} if
\begin{enumerate}
\item $\bk\to \End_{\cC}(\unit)$ is an isomorphism;
\item $(\cC,\otimes,\unit)$ is rigid, so every object has a left and right dual;
\item $\cC$ is abelian, and all objects have finite length.
\end{enumerate}
It follows that $\unit$ is simple, see \cite[Theorem~4.3.8]{EGNO}.
It then further follows that morphism spaces in $\cC$ are finite dimensional, since $\Hom_{\cC}(X,Y)$ is isomorphic to $\Hom_{\cC}(\unit, Y\otimes X^\ast)$.

We will say that a tensor category is {\bf quasi-finite} if it is a union of tensor categories with finitely many simple objects (up to isomorphism). Then we have a well-defined notion of Frobenius-Perron dimension, see \cite[Proposition~4.5.7]{EGNO}. A {\bf tensor functor} between two tensor categories over $\bk$ is a monoidal functor which is $\bk$-linear and exact (and consequently also faithful). 

We say that $X\in\cC$ {\bf generates} $\cC$ as a tensor category if every object in $\cC$ is a subquotient of a direct sum of tensor products of $X$ and its (iterated) duals.

\subsubsection{Growth dimension}
For a tensor category $\cC$ and $X\in\cC$, we set
$$\gd(X)=\lim_{i\to\infty}\ell(X^{\otimes i})^{1/i}\,\in\, \mR\cup\{\infty\},$$
where the limit is well-defined, see for instance \cite{CEO1, CEO3, EK}. We say that X is of {\bf moderate growth} if $\gd(X)<\infty$ and that $\cC$ is of moderate growth if every object is of moderate growth.
In case $\cC$ is quasi-finite, then
$$\gd(X)\,=\, \FPdim(X),$$ by \cite[Lemma~8.3]{CEO1}.

If $\cC$ is of moderate growth and $\gd(X\oplus Y)=\gd(X)+\gd(Y)$ for all $X,Y\in\cC$, then $\gd$ defines unambiguously an additive function $K_0(\cC)\to\mR$. Even without this assumption, we do have

\begin{lemma}\label{Lemgd}
Consider $Y\in\cC$ of moderate growth.
\begin{enumerate}
\item Let $f(x)\in\mZ[x]$ be a polynomial with non-negative coefficients. For $V\in\cC$ with $[V]=f([Y])$ in $K_0(\cC)$, we have
$$\gd(V)\,=\, f(\gd(Y)).$$
\item Let $g(x)\in\mZ[x]$ be arbitrary.
\begin{enumerate}
\item If $g([Y])\in K_0(\cC)$ equals the class of an non-zero object in $\cC$, then $g(\gd(Y))>0$.
\item If $g([Y])=0$, then $g(\gd(Y))=0$.
\end{enumerate}
\end{enumerate}
\end{lemma}
\begin{proof}
For part (1), we start by observing that, for all $A,B\in\cC$ and $i\in\mN$ we have
\begin{enumerate}
\item[(i)] $\gd(A\oplus B)\ge \gd(A)+\gd(B)$;
\item[(ii)] $\gd(A^{\otimes i})=\gd(A)^i$;
\item[(iii)] $\ell(A^{\otimes i})\le \gd(A)^i$.
\end{enumerate}
Indeed, (i) and (ii) are immediate, and (iii) follows from (ii). That $\gd(V)\ge f(\gd(Y))$ follows from (i) and (ii). From (iii) we can derive that $\ell(V^{\otimes n})\le f(\gd Y)^n$, for all $n$, from which $\gd(V)\le f(\gd(Y))$ follows.


Part (2) is an application of part (1). Indeed, write $g=g_+-g_-$ for polynomials $g_{\pm}$ with non-negative coefficients. Define $V_{\pm}\in\cC$ as the semisimple objects with $[V_{\pm}]=g_{\pm}([X])$. By part (1)
$$g(\gd(Y) )\,=\, \gd(V_+)-\gd(V_-).$$ 
Under assumption (a), $V_-$ is a proper subobject of $V_+$, which implies that $\gd(V_-)<\gd(V_+)$. Under assumption (b), $V_-\simeq V_+$.
\end{proof}

\subsubsection{Symmetric tensor categories}
Let $\cC$ be a symmetric (= equipped with a symmetric braiding $\sigma$) tensor category.

If, for $X\in\cC$, the object $X^\ast\otimes X$ is of moderate growth, then the tensor subcategory of $\cC$ generated by $X$ is of moderate growth. For instance, we then have
\begin{equation}
\label{DefDX}\gd(X)\,=\,\gd(X^\ast)\,\le\, \sqrt{\gd(X^\ast\otimes X)}\,=:\, \bD(X).
\end{equation}

For $X\in\cC$, we let $\wedge^2X$ denote the image of the endomorphism $1-\sigma_{X,X}$ of $X^{\otimes 2}$, where $\sigma$ is the braiding. Whenever $\mathrm{char}(\bk)\not=2$, it is more common to define $\wedge^2X$ as the kernel of $1+\sigma_{X,X}$, which is then equivalent. We define the symmetric powers of $X$ via the short exact sequence
$$0\to \sum_{i+j=n-2} X^{\otimes i}\otimes \wedge^2 X\otimes X^{\otimes j}\to X^{\otimes n}\to\Sym^n X\to 0.$$
We define the (dual or divided) exterior powers as subobjects
$$\wedge^nX = \bigcap_{i+j=n-2}X^{\otimes i}\otimes \wedge^2 X\otimes X^{\otimes j}\subset X^{\otimes n}.$$

\subsubsection{Higher Verlinde categories}
Let $\bk$ be an algebraically closed field of characteristic $p>0$. We refer to \cite{BEO} for the definition and conventions regarding the finite symmetric tensor cateogries $\Ver_{p^n}$ over $\bk$ for $n\in\mZ_{>0}$. In particular, we follow the labelling of simple objects from \cite{BEO} where $L_0=\unit$. This deviates from the convention in \cite{Os} for $n=1$.

\subsection{Chebyshev polynomials}\label{SecCheb}

\subsubsection{}\label{DefCheb}
For $n\in\mN$, define the polynomials $\kappa_n(x)\in\mZ[x]$ by
$$\kappa_0(x)=1,\quad\kappa_1(x)=x,\quad\mbox{and}\quad \kappa_{n+1}(x)=x\kappa_n(x)-\kappa_{n-1}(x).$$
In other words, $\kappa_n(x) $ is $ U_n(x/2),$
with $U_n$ the Chebyshev polynomials of the second kind. In particular
$$\kappa_{n}(x)\;=\;\sum_{i=0}^{\lfloor n/2\rfloor}(-1)^i\binom{n-i}{i}x^{n-2i},\quad\mbox{and}\quad\kappa_n(y+y^{-1})=\frac{y^{n+1}-y^{-n-1}}{y-y^{-1}}.$$
For example
$$\kappa_2(x)=x^2-1,\;\kappa_3(x)=x^3-2x,\; \kappa_4(x)=x^4-3x^2+1,\; \kappa_5(x)=x^5-4x^3+3x.$$

We use the same notation $\kappa_n$ when we consider the polynomials in $R[x]$ for another commutative ring $R$.
\begin{lemma}\label{Lem4}
\begin{enumerate}
\item Let $p$ be a prime. In $\mathbb{F}_p[x]$, we have
$$\kappa_{p^l-1}(x)\;=\;\begin{cases} (x^2-4)^{\frac{p^l-1}{2}}& \mbox{ if } p>2\\
x^{2^l-1}&\mbox{ if }p=2.
\end{cases}$$
\item Let $N\in\mZ_{>0}$. In $\mR[x]$, we have
$$\kappa_{N-1}(x)\;=\; \prod_{1\le j<N }\left(x-2\cos\left(\frac{j\pi}{N}\right)\right).$$
\end{enumerate}

\end{lemma}
\begin{proof}
For part (1), we prove the case $p>2$. Over $\mathbb{F}_p$, we have
$$\kappa_{p^l-1}(y+y^{-1})\;=\; \frac{y^{p^l}-y^{-p^l}}{y-y^{-1}}\;=\; (y-y^{-1})^{p^l-1}\;=\; ((y+y^{-1})^2-4)^{\frac{p^l-1}{2}},$$
from which the claim follows. Part (2) is well-known.
\end{proof}

\subsubsection{} We also consider the two-variable version $\mu_n(x,y)\in\mZ[x,y]$, symbolically defined by
$$\mu_{2i}(x,y)=\kappa_{2i}(\sqrt{xy}),\quad\mu_{2i+1}(x,y)=\sqrt{\frac{x}{y}}\;\kappa_{2i+1}(\sqrt{xy}).$$

\subsubsection{Cyclotomic part} It is clear that, if there exist $\nu_i(x,y)\in\mZ[x,y]$, $i\in\mN$, with the property that for all $n\in\mZ_{>0}$
$$\mu_{n-1}(x,y)\;=\;\prod_{i|n} \nu_{i-1}(x,y),$$
then these relations determine them uniquely. Indeed, the polynomials $\nu_i(x,y)$ exist, see \cite[Lemma~3.2]{Ha}. For example, $\nu_{N-1}(x,x)$ is given by the product in Lemma~\ref{Lem4}(2), where~$j$ is now restricted by the condition $\gcd(j,N)=1$.


\subsection{Lopsided Temperley-Lieb algebras and Jones-Wenzl idempotents}

In this section we fix a triple
$$\bR\;:=\; (R,\delta_1,\delta_2)$$
of a commutative ring $R$ and two elements $\delta_1,\delta_2\in R$. We also set $\bR'=(R,\delta_2,\delta_1)$.

\subsubsection{Quantum numbers}
We set $[0]_{\bR}=0\in R$ and
$$[n]_{\bR}\;=\; \mu_{n-1}(\delta_1,\delta_2)\in R,\quad\mbox{for $n\in\mZ_{>0}$}.$$
Similarly, we set $[[0]]_{\bR}=0\in R$ and
$$[[n]]_{\bR}\;=\; \nu_{n-1}(\delta_1,\delta_2)\in R,\quad\mbox{for $n\in\mZ_{>0}$}.$$

More generally, we set (symbolically)
$$\binom{n}{i}_{\bR}\;=\; \frac{[n]_{\bR}[n-1]_{\bR}\cdots [n-i+1]_{\bR}}{[i]_{\bR}[i-1]_{\bR}\cdots[1]_{\bR}},\quad\mbox{for }\; 1\le i\le n,$$
which yields well-defined elements in $R$. Indeed, we can always rewrite $\binom{n}{i}_{\bR}$ as a product of the elements~$[[j]]_{\bR}$.
\begin{example}\label{ExQN}
\begin{enumerate}
\item We have $[1]_{\bR}=1$, $[2]_{\bR}=\delta_1$, $[3]_{\bR}=\delta_1\delta_2-1$ and $[4]_{\bR}=\delta_1(\delta_1\delta_2-2)$.
\item In case there is $q\in R^\times$ with
$$\delta_1\;=\;\delta_2\;=\; q+q^{-1},$$
we find the familiar
$$[n]_{\bR}\;=\; \kappa_{n-1}(q+q^{-1})\;=\; \frac{q^n-q^{-n}}{q-q^{-1}}.$$
In case $q=\pm 1$, this is to be interpreted as $[n]_{\bR}=(\pm 1)^{n+1}n$. 
\item We have
$$\binom{4}{2}_{\bR}\;=\; [[4]]_{\bR}[3]_{\bR}=(\delta_1\delta_2-2)(\delta_1\delta_2-1).$$
\end{enumerate}
\end{example}

\subsubsection{Temperley-Lieb algebra}
For every $n\in\mN$, we have the associated lopsided (or two-coloured) Temperley-Lieb algebra $\TL_n(\bR)$, see \cite{El, EW, Ha}. This is an $R$-algebra, free as an $R$-module of rank the $n$-th Catalan number $C_n$, typically considered with a basis of non-crossing diagrams.
 We take the convention that in a product of diagrams, the left one goes on top.

It will be convenient to decorate the odd, when counting from the right, upper points of lines as $\uparrow$ and the even lower points as $\downarrow$. In particular, the diagram basis of $\TL_3(\bR)$ is 
$$
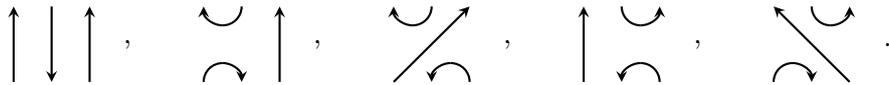
\begin{tikzpicture}[scale=0.5,thick,>=angle 90]
\begin{scope}[xshift=4cm]

\draw[-stealth]  (0,-1) -- +(0,2);

\draw[stealth-]  (1,-1) -- +(0,2);

\draw[-stealth]  (2,-1) -- +(0,2);

\draw (3,0) node[]{,};

\draw  (5,-1) to [out=90, in=180] +(0.5,0.5);
\draw[stealth-] (6,-1) to [out=90, in=0] +(-0.5,0.5);
\draw[stealth-]  (5,1) to [out=-90, in=180] +(0.5,-0.5);
\draw (6,1) to [out=-90, in=0] +(-0.5,-0.5);
\draw[-stealth]  (7,-1) -- +(0,2);
\draw (8,0) node[]{,};

\draw[stealth-]  (11,-1) to [out=90, in=180] +(0.5,0.5);
\draw (12,-1) to [out=90, in=0] +(-0.5,0.5);
\draw[stealth-]  (10,1) to [out=-90, in=180] +(0.5,-0.5);
\draw (11,1) to [out=-90, in=0] +(-0.5,-0.5);

\draw[-stealth] (10,-1) -- +(2,2);

\draw (13,0) node[]{,};

\draw[stealth-]  (16,-1) to [out=90, in=180] +(0.5,0.5);
\draw (17,-1) to [out=90, in=0] +(-0.5,0.5);
\draw  (16,1) to [out=-90, in=180] +(0.5,-0.5);
\draw[stealth-] (17,1) to [out=-90, in=0] +(-0.5,-0.5);


\draw[-stealth]  (15,-1) -- +(0,2);
\draw (18,0) node[]{,};

\draw[-stealth] (22,-1) -- +(-2,2);
\draw  (21,1) to [out=-90, in=180] +(0.5,-0.5);
\draw[stealth-] (22,1) to [out=-90, in=0] +(-0.5,-0.5);
\draw (23,0) node[]{.};
\draw  (20,-1) to [out=90, in=180] +(0.5,0.5);
\draw[stealth-] (21,-1) to [out=90, in=0] +(-0.5,0.5);

\end{scope}
\end{tikzpicture}
$$

Composition of diagrams is defined by concatenation with counter-clockwise loops evaluated at $\delta_1$ and clockwise loops at $\delta_2$. The similar diagrammatic approach from \cite{El} uses two colours (as opposed to arrows) for the areas separated by the lines and lets the colour of the interior of a loop determine its evaluation. 

When $\delta_1=\delta_2$, the decoration becomes immaterial, and we recover an ordinary Temperley-Lieb algebra. 
We also have an algebra isomorphism
$$\TL_n(R,\delta_1,\delta_2)\;\xrightarrow{\sim}\; \TL_n(R,\lambda\delta_1,\lambda^{-1}\delta_2),$$
for each $\lambda\in R^\times$, which sends each diagram to a scalar multiple, with scalar given by $\lambda^{a-b}$, where $a$ is the number of clockwise caps and $b$ the number of counter-clockwise caps.


\begin{definition} \label{DefJW}
The {\bf Jones-Wenzl idempotent} $\JW_n\in \TL_n(\bR)$, with $n\in\mZ_{>0}$, is (when it exists) the unique idempotent which yields zero when multiplying on the left (or equivalently on the right) with each diagram with a cup, and for which the coefficient of the identity diagram basis element is $1$. 
\end{definition}
By convention, we let $\JW_1$ be the identity. We will usually abbreviate the condition that the Jones-Wenzl idempotent exists in $\TL_n(\bR)$ to saying that ``$\bR$ is $\mathcal{J}\mathcal{W}_n$''. We will also say that~$\bR$ is $\cJW_n'$ when $\bR'$ is $\cJW_n$.
The following explicit characterisation of the $\cJW_n$ condition is, for $\delta_1=\delta_2$, due to Webster, and in general due to Hazi, see \cite[Theorem~A]{Ha}.

\begin{prop}(Hazi)\label{PropHazi}
The idempotent $\JW_n$ exists in $\TL_n(\bR)$ if and only if $\binom{n}{i}_{\bR}$ is invertible in $R$ for each $1\le i\le n$.
\end{prop}



We are mainly interested in the existence of the Jones-Wenzl idempotent $\JW_n$ under the additional condition $[n+1]_{\bR}=0$.
\begin{example}\label{ExEx} Assume that $R$ is a field $\bk$ and that $\delta_1=\delta_2=\delta$.

\begin{enumerate}
\item If $n=p^l-1$, for $p=\mathrm{char}(\fk)>0$, then $[n+1]_{\bR}=0$ is equivalent to $\delta=\pm 2$, by Lemma~\ref{Lem4}(1). In this case, $\JW_n$ always exists, as follows from Lukas' theorem and Example~\ref{ExQN}(2). 
\item If $\mathrm{char}(\fk)=0$ and $[n+1]_{\bR}=0$ then $\JW_n$ exists in $\TL_n(\bR)$. Indeed, from Lemma~\ref{Lem4}(2) we can derive that there is at most one $i$ for which $[[i]]_{\bR}=0$, from which we can derive the claim.
\item If $\mathrm{char}(\fk)=2$ and $\delta=0$, then $\JW_5$ does not exist in $\TL_5(\bR)$, even though $[6]_{\bR}=0$. Indeed,
$\binom{5}{2}_{\bR}=0.$
\end{enumerate}
\end{example}

The `rotatability' of the Jones-Wenzl idempotents is considered in \cite{El, Ha}. We say that~$\bR$ is $\cJWR_n$ if $\TL_n(\bR)$ and $\TL_n(\bR')$ both admit the Jones-Wenzl idempotent and they are rotatable, {{\it i.e.} one `rotates' into the other. We summarise some results:

\begin{prop}(Hazi)\label{PropRot0}
The following conditions are equivalent on a triple $\bR=(R,\delta_1,\delta_2)$ and $n\in \mZ_{>1}$.
\begin{enumerate}
\item $\bR$ is $\cJWR_n$.
\item For all $1\le i\le n$, we have $$\binom{n+1}{i}_{\bR}\;=\;0\;=\; \binom{n+1}{i}_{\bR'}.$$
\item $\bR$ is $\cJW_n$ and for every $1\le k\le n$
$$\prod_{l}[[l]]_{\bR}\;=\;0$$
where $l$ runs over all divisors of $n+1$ that do not divide $k$.
\end{enumerate}
\end{prop}
\begin{proof}
The equivalence of (1) and (2) is \cite[Theorem~B]{Ha}. The equivalence of (1) and (3) is a reformulation of \cite[Proposition~4.7]{Ha}, using the proof of \cite[Lemma~4.6]{Ha}.
\end{proof}

\begin{corollary}\label{CorRot}If~$R$ is an integral domain, then $\bR$ is $\cJWR_n$ if and only if
$$[[n+1]]_{\bR}\;=\;0\;=\; [[n+1]]_{\bR'}.$$
\end{corollary}
\begin{proof}
That the two vanishing conditions imply the conditions in Proposition~\ref{PropRot0}(2) follows from the expansion of the binomial coefficients in terms of the $[[j]]_{\bR}$.

Conversely, assume that $\bR$ is $\cJWR_n$. Set
$$m\;:=\;\max\{ i\in[1, n+1]\,|\, [[i]]_{\bR} =0\}.$$
Note that this is well-defined since, by Proposition~\ref{PropRot0}(2)
$$0\;=\;\binom{n+1}{1}_{\bR}\;=\;[n+1]_{\bR}\;=\; \prod_{l|n+1}[[l]]_{\bR}.$$
Assume for a contradiction that $m<n+1$. By Proposition~\ref{PropRot0}(3) there exists an $l\le n$ that does not divide $m$ with $[[l]]_{\bR}=0$. By maximality of $m$ we also know that $m$ does not divide $l$. But \cite[Lemma~3.4]{Ha} then implies that $[[m]]_{\bR}$ and $[[l]]_{\bR}$ cannot both be zero, a contradiction. The identical argument for $\bR'$ concludes the proof.
\end{proof}

%
%
%


\section{Additive case}\label{SecAdd}

Let $\cA$ be an additive monoidal category (with bi-additive tensor product). The main aim of this section is unpacking the following definition, and its consequences, purely within the language of monoidal categories.

\begin{definition}
A rigid object $X\in\cA$ is {\bf homotopically $N$-bounded}, with $N\in\mZ_{>0}$, if the endofunctor $-\otimes X$ of the stable $\infty$-category $\mathrm{N}_{\mathrm{dg}}(\mathrm{Ch}(\cA))$, see \cite[Proposition~1.3.2.10]{Lu}, is $N$-spherical in the sense of \cite[Definition~4.1.1]{DKS}.
\end{definition}

We also say that $X$ is {\bf homotopically bounded}, if it is homotopically $N$-bounded for some $N\in\mZ_{>0}$ and {\bf homotopically unbounded} otherwise.

\subsection{Continuant complexes}
Fix a rigid $X\in\cA$.

\subsubsection{}\label{PrepDefComp} We consider 
$$\mE_0(X)\;:=\; \unit[0]\quad\mbox{and}\quad \mE_1(X):=X[0]$$
as objects in $\Ch(\cA)$ concentrated in homology degree $0$. We have the `chain map'
$$f_1:=\; X^\ast \otimes \mE_1(X) \xrightarrow{\ev_X} \mE_0(X). $$
Justified by \cite[equation~(3.3.2)]{DKS}, we continue as below.

\begin{definition}\label{DefComp}
We define $\mE_n(X)\in\Ch(\cA)$ and chain maps $f_n$ for $n\in\mZ_{>1}$ iteratively as follows.
\begin{enumerate}
\item The complex $\mE_{n}(X)$ is given by $\Cone (f_{n-1})[-1]$.
\item The chain map
$$f_{n}:X^{(n)}\otimes \mE_{n}(X)\;\to\; \mE_{n-1}(X)$$
is defined by adjunction from the canonical chain map
$$\phi_{n}:\mE_{n}(X)=\Cone(f_{n-1})[-1]\;\to\; X^{(n-1)}\otimes \mE_{n-1}(X).$$
\end{enumerate}
\end{definition}
We also let $\phi_1$ be the defining isomorphism $\mE_1(X)\xrightarrow{\sim}X[0]$.
By definition, we have exact triangles in $\cK(\cA)$ for $n>1$:
\begin{equation}\label{DefTria}
\mE_n(X)\;\xrightarrow{\phi_n}\; X^{(n-1)}\otimes \mE_{n-1}(X)\;\xrightarrow{f_{n-1}}\; \mE_{n-2}(X)\;\to\;.
\end{equation}

\begin{example}With obvious abbreviation of morphisms:
\begin{enumerate}
\item The complex $\mE_2(X)$ is 
$\; 0\to X^\ast\otimes X\xrightarrow{\ev_X} \unit\to 0$.
\item The complex $\mE_3(X)$ is 
$$ 0\to X^{\ast\ast}\otimes X^\ast\otimes X\xrightarrow{\left(\begin{array}{c}
\ev_{X}\\
\ev_{X^\ast}
\end{array}\right)} X^{\ast\ast}\oplus X\to 0.$$
\item The complex $\mE_4(X)$ is
$$ 0\to X^{\ast\ast\ast}\otimes X^{\ast\ast}\otimes X^\ast\otimes X\xrightarrow{\left(\begin{array}{c}
\ev_X\\
\ev_{X^\ast}\\
\ev_{X^{\ast\ast}}
\end{array}\right)}(X^{\ast\ast\ast}\otimes X^{\ast\ast})\oplus (X^{\ast\ast\ast}\otimes X) \oplus (X^\ast\otimes X) \xrightarrow{(\ev_{X^{\ast\ast}}, 0,-\ev_{X})} \unit\to 0 .$$
\end{enumerate}
The term $\cdots\otimes X^\ast\otimes X$ is always in homological degree $0$.
\end{example}

\begin{remark}\label{RemExp}
In general, it follows by construction, or from the detailed analysis in \cite[\S 3]{DKS}, that $\mE_n(X)$ is a complex of the form
\begin{equation}\label{eqAI}
0\to A\to \bigoplus_{|I|=2}A_I\to\bigoplus_{|J|=4}A_J\to\cdots,
\end{equation}
where the summations run over {\em twinned} subsets of $\{0,1,2,\cdots,n-1\}$ of the given cardinality. Twinned subsets are disjoint unions of neighbouring pairs $\{i,i+1\}$. Furthermore, 
$$A\;:=\; X^{(n-1)}\otimes X^{(n-2)}\otimes\cdots X^\ast\otimes X,$$
and for a twinned subset $I$, the object $A_I$ is obtained from $A$ by removing all factors $X^{(i)}$, $i\in I$, in the product. For twinned subsets $I,J\subset\{0,1,2,\cdots,n-1\}$, we write $I\lhd J$ to denote that $I\subset J$ such that $J\backslash I$ is of the form $\{i,i+1\}$. Whenever $I\lhd J$ we have a morphism
$$\ev^I_J\;:\; A_I\to A_J$$
given by whiskering $\ev_{X^{(i)}}$. The morphisms in the complex \eqref{eqAI} are matrices with entries given by these evaluations, up to sign, or zero for all remaining cases.
\end{remark}

\subsubsection{} We can provide a similar iterative definition of complexes, $\mE^0(X)=\unit[0]$, $\mE^1(X)=X[0]$ and
$$\mE^n(X)\;:=\;\Cone\left(\mE^{n-2}(X)\,\to\, \mE^{n-1}(X)\otimes X^{(n-1)}\right).$$
More directly, we can define these complexes as
\begin{equation}\label{eqNN} \mE_n({}^\ast X)^\ast\;\simeq\; \mE^n(X)\;\simeq\; {}^\ast\mE_n(X^\ast).\end{equation}
That the expressions on the left and right actually agree can be proved iteratively from Definition~\ref{DefComp}.

\begin{remark}\label{RemCont}
As observed in \cite{DKS}, in the (split) Grothendieck group of $\cA$ we have
$$[\mE_n(X)]=E_n([X^{(n-1)}],[X^{(n-2)}],\cdots, [X])\;\mbox{ and }\;[\mE^n(X)]=E_n([X],[X^\ast],\cdots,[X^{(n-1)}]),$$
with $E_n$ Euler's (signed) continuant polynomial in $\mZ\langle x_1,x_2,\cdots, x_n\rangle$. Our interest in the polynomials in Section~\ref{SecCheb} stems from
$$E_n(x,x,\cdots,x)=\kappa_n(x)\quad\mbox{and}\quad E_n(\cdots, x,y,x)=\mu_n(x,y),$$
for commuting $x,y$, see for instance \cite[Proposition~1.4.1]{DKS}.
\end{remark}

\subsection{Some exact triangles}

In this section we use the octahedral axiom to construct new exact triangles from the triangles \eqref{DefTria} defining $\mE_n(X)$ and $\mE^n(X)$. We start by constructing a family of complexes, as a variation of Definition~\ref{DefComp}, which will actually turn out to yield isomorphic complexes.

 \subsubsection{}\label{DefD}
We set $\mD_0(X)=\mE_0(X)=\unit[0]$ and $\mD_1(X)=\mE_1(X)=X[0]$. We view evaluation as the chain map
$$t_1: \mD_1(X^\ast) \otimes X\xrightarrow{\ev_X} \mD_0(X^{\ast\ast}).$$

We define $\mD_n(X)\in\Ch(\cA)$ and $t_n$ for $n\in\mZ_{>1}$ iteratively as follows.
\begin{enumerate}
\item The complex $\mD_{n}(X)$ is given by $\Cone (t_{n-1})[-1]$.
\item The chain map
$$t_{n}:\mD_{n}(X^\ast)\otimes X\;\to\; \mD_{n-1}(X^{\ast\ast})$$
is defined by adjunction from the canonical chain map
$$\theta_{n}:\mD_{n}(X^{\ast})=\Cone(t_{n-1})[-1]\;\to\; \mD_{n-1}(X^{\ast\ast})\otimes X^{\ast}.$$
\end{enumerate}
We also let $\theta_1$ denote the identity $\mD_1(X)=X[0]$.

\begin{remark}
\begin{enumerate}
\item One can verify directly that $\mD_n(X)=\mE_n(X)$ for $n\le 3$.
\item Contrary to the definition of $\mE_n(X)$, the definition of $\mD_{n}(X)$ involves $\mD_i(X^{(j)})$ for $j\not=0$. We continue to take the convention to use the symbols $t_n$ and $\theta_n$ without reference to the specific object ($X$ or some dual) in order to keep notation lighter.
\end{enumerate}
\end{remark}

\subsubsection{}\label{Oct} We will use the following consequence of the octahedral axiom. Given composable morphisms $a: U\to V$, $b:V\to W$ in a triangulated category, there exist an exact triangle and a commutative triangle
$$\xymatrix{ \Cone(c)[-1]\ar[rr]&& \Cone(b)[-1]\ar[rr]\ar[rd]&&\Cone(a)\ar[rr]&&\\
&&&V\ar[ru]
}$$
where the morphisms to and out of $V$ define the cones.

 \begin{prop}\label{ThmTrian}For every $2\le l\le n$,
there is an exact triangle in $\cK(\cA)$:
 $$\mE_n(X)\;\to\; \mE_{l-1}(X^{(n-l+1)})\otimes \mE_{n-l+1}(X)\;\to\; \mE_{l-2}(X^{(n-l+2)})\otimes \mE_{n-l}(X)\;\to\;.$$
 \end{prop}
\begin{proof}
For $l=2$, we can take the defining triangle \eqref{DefTria}.
More generally, for $2\le l\le n$, we claim that there is a commutative diagram
$$\xymatrix{
&\mD_{l-1}(X^{(n-l+1)})\otimes X^{(n-l)}\otimes \mE_{n-l}(X)\ar[rd]^{t_{l-1}\otimes\mE_{n-l}(X)}&&\\
\mE_n(X)\ar[r]& \mD_{l-1}(X^{(n-l+1)})\otimes \mE_{n-l+1}(X)\ar[r]\ar[d]_{\theta_{l-1}\otimes\mE_{n-l+1}(X)}\ar[u]^{\mD_{l-1}(X^{(n-l+1)})\otimes\phi_{n-l+1}}&\mD_{l-2}(X^{(n-l+2)})\otimes\mE_{n-l}(X)\\
&\mD_{l-2}(X^{(n-l+2)})\otimes X^{(n-l+1)}\otimes \mE_{n-l+1}(X),\ar[ru]_{\mD_{l-2}(X^{(n-l+2)}) \otimes f_{n-l+1}}&
}$$
where the middle row is part of an exact triangle.
As $\theta_1=\id$, the case $l=2$ indeed reduces to the previous triangle. That the diagram is commutative follows immediately from the definitions of the maps $f$ and $t$ via adjunction. 

That the exact triangle on the middle row exists follows by iteration on $l$, with the already established base case $l=2$. Indeed, taking the upper triangle for the case $l$, and applying~\ref{Oct} gives the exact triangle with lower commutative triangle for $l+1$.

Finally, we focus on the case $l=n$, in which case we obtain an exact triangle
$$\mE_n(X)\;\to\;\mD_{n-1}(X^\ast)\otimes X\;\xrightarrow{t_{n-1}}\; \mD_{n-2}(X^{\ast\ast}).$$
By Definition in \ref{DefD}, we thus find $\mE_n(X)\simeq\mD_n(X)$, for all $n$. Making this identification in the middle row of the commutative diagram thus concludes the proof.
\end{proof}

Part (1) of the following proposition is essentially a special case of \cite[Proposition~3.3.7]{DKS}. We repeat the proof {\it loc. cit.} since we need to keep track of some morphisms for applications below. 
\begin{prop}\label{PropDKS}
For every $n\in\mZ_{>0}$ there are exact triangles in $\cK(\cA)$:
\begin{enumerate}
\item $\mE^{n-1}(X^\ast)\otimes \mE_{n+1}(X)\to \mE^n(X^\ast)\otimes \mE_n(X)\to\unit\to\;,$
\item $\mE_{n+1}(X)\otimes \mE^{n-1}(X^\ast)\to \mE_n(X^\ast)\otimes\mE^n(X)\to\unit\to\;,$
\end{enumerate}
where the arrows to $\unit$ correspond to evaluation, using the interpretations $\mE_n(X)^\ast\simeq \mE^n(X^\ast)$ and $\mE^n(X)^\ast\simeq\mE_n(X^\ast)$ from~\eqref{eqNN}.
\end{prop}

\subsubsection{}\label{Oct0} We use the following consequence of the octahedral axiom: Given composable morphisms $a: U\to V$, $b:V\to W$ in a triangulated category, there exists an exact triangle with commutative diagram
$$\xymatrix{  \Cone(b)[-1]\ar[rr]\ar[rd]&&\Cone(a)\ar[rr]&&\Cone(c)\ar[rr]&&\\
&V\ar[ru]\ar[rr]^b&& W\ar[ru]
}$$
where the morphisms to and out of $V$ and the one out of $W$ define the cones.

\begin{proof}[Proof of Proposition~\ref{PropDKS}]
The first part of the proof of (1) can literally be taken from~\cite{DKS}. The case $n=1$ is simply the definition $\mE_2(X)=\Cone(f_1)[-1]$.
The higher cases can then be obtained by iteration.
Indeed, we have a commutative square
\begin{equation}\label{sqfphi}\xymatrix{
&\mE^n(X^\ast)\otimes X^{(n+1)}\otimes \mE_{n+1}(X)\ar[rd]^{\mE^n(X^\ast)\otimes f_{n+1}}&\\
\mE^{n-1}(X^\ast)\otimes \mE_{n+1}(X)\ar[rd]^{\mE^{n-1}(X^\ast)\otimes \phi_{n+1}}\ar[ru]^{f_n^\ast\otimes \mE_{n+1}(X)}&&\mE^n(X^\ast)\otimes \mE_{n}(X).\\
&\mE^{n-1}(X^\ast)\otimes X^{(n)}\otimes\mE_{n}(X)\ar[ru]^{\phi^\ast_n\otimes \mE_n(X)}
}\end{equation} In particular, the already established case $n=1$ of (1) has the composite as its first arrow. We will prove by iteration that triangle (1) exists, with first arrow the composite of~\eqref{sqfphi} and with second arrow evaluation.

Starting from triangle (1), we split up the first arrow according to the upper path in \eqref{sqfphi}. The defining triangles of $\mE_i(X)$ and $\mE^i(X^\ast)$ and \ref{Oct0} then yield the triangle in (1) with $n$ replaced by $n+1$ and first arrow as claimed.  

To focus on the second arrow, we observe that by \ref{Oct0} (and induction), the triangle for $n+1$ can be taken such that its second arrow fits into the commutative square
$$\xymatrix{
&\mE^{n+1}(X^\ast)\otimes\mE_{n+1}(X)\ar[rd]&\\
\mE^n(X^\ast)\otimes X^{(n+1)}\otimes \mE_{n+1}(X)\ar[ru]^{\phi_{n+1}^\ast\otimes \mE_{n+1}(X)}\ar[rd]^{\mE^n(X^\ast)\otimes f_{n+1}}&&\unit.\\
&\mE^{n}(X^\ast)\otimes\mE_n(X)\ar[ru]^{\ev}\ar[ru]
}$$
Now in general, for instance in the universal rigid additive monoidal category on one object \cite{CSV}, it is easy to see that the only endomorphisms of $\mE_{n+1}(X)$ (in $\Ch(\cA)$ or $\cK(\cA)$) are scalar multiples of the identity. Hence, we only need to determine for which scalar multiple of the evaluation of $\mE_{n+1}(X)$, the latest square is commutative. Since the left upwards arrow comes from an adjunction of the left downwards arrow, it follows this scalar is simply 1.

The triangle in part (2) is obtained identically, but using the complexes $\mD_n(X)$ (and their duals) rather than $\mE_n(X)$. Indeed, $\mD_n(X)\simeq \mE_n(X)$ by the proof of \autoref{ThmTrian}.
\end{proof}


\begin{corollary}\label{CorTria}
If $\mE_{N-1}(X)= 0$ in $\cK(\cA)$ for $X\in\cA$ and $N\in\mZ_{>1}$, then in $\cK(\cA)$:
\begin{enumerate}
\item $\mE_{N-2}(X)\simeq\mE_N(X)[1]$;
\item For $i\in\{N,N-2\}$, evaluation yields isomorphisms
$$\mE_i(X)^\ast \otimes \mE_i(X)\xrightarrow{\sim}\unit\quad\mbox{and}\quad  \mE^i(X)^\ast\otimes\mE^i(X)\xrightarrow{\sim}\unit;$$
\item For $m\in\mN$,
\begin{enumerate}
\item $\mE_{m+N}(X)[1]\;\simeq\; \mE_{N-2}(X^{(m+2)})\otimes \mE_{m}(X)$, if $m$ is odd;
\item $\mE_{m+N}(X^\ast)[1]\;\simeq\; \mE_{N-2}(X^{(m+3)})\otimes \mE_{m}(X^\ast)$, if $m$ is even;
\item $\mE_{m+N}(X)[1]\;\simeq\; \mE_{m}(X^{(N)})\otimes\mE_{N-2}(X)$.
\end{enumerate}

\end{enumerate}
\end{corollary}
\begin{proof}
Part (1) follows from the defining triangle~\eqref{DefTria}. Part (2) follows from Proposition~\ref{PropDKS}. Part (3) consists of special cases of \autoref{ThmTrian}, using $\mE_{N-1}(X^{(2i)})\simeq\mE_{N-1}(X)^{(2i)}$.
\end{proof}

\subsection{Bounded objects}
In practice, we will use the following characterisations as definition of homotopical boundedness:
\begin{prop}\label{PropSph}For $N\in \mZ_{>0}$, the following conditions on the rigid object $X\in\cA$ are equivalent:
\begin{enumerate}
\item $X$ is homotopically $N$-bounded;
\item 
$\mE_{N-1}(X)\simeq 0\simeq \mE^{N-1}(X)\;$ in $\;\cK(\cA)$;
\item $\mE_{N-1}(X)\simeq 0\simeq \mE_{N-1}(X^\ast)\;$ in $\;\cK(\cA)$;
\item $\mE_{N-1}(X^{(i)})\simeq 0 \;$ in $\;\cK(\cA),\;$ for all $i\in\mZ.$
\end{enumerate}
\end{prop}
\begin{proof}
The definition of sphericality in \cite[Definition~4.1.1]{DKS} is given in terms of vanishing of certain (exact) functors on stable $\infty$-categories. These vanish if and only if they vanish on the homotopy category. By our mimicking of the constructions in \cite{DKS}, the corresponding endofunctors on
$$\mathrm{Ho}(N_{\mathrm{dg}}(\Ch(\cA)))\;\simeq\; \cK(\cA)$$
are $-\otimes \mE_{N-1}(X)$ and $-\otimes \mE^{N-1}(X)$. Clearly, they vanish if and only if (2) is satisfied.

Equivalence of (2), (3) and (4) follows immediately from equation~\eqref{eqNN}.
\end{proof}

\begin{example}
\begin{enumerate}
\item No objects are homotopically 1-bounded.
\item The homotopically 2-bounded objects are precisely the zero objects.
\item The homotopically 3-bounded objects are precisely the invertible objects.
\end{enumerate}
\end{example}

\begin{remark}
\begin{enumerate}
\item Both conditions in Proposition~\ref{PropSph}(2) are needed. Indeed, this is the case already for $N=3$, as there exist `left invertible' or `right invertible' objects which are not invertible.
\item In \cite[\S 5.5]{DKS}, the notion of an `$N$-spherical object' is introduced, generalising the notion of (4-)spherical objects in triangulated categories due to Seidel and Thomas. We therefore do not use the terminology (homotopically) `$N$-spherical objects' for what we call (homotopically) `$N$-bounded objects'.
\item The choice of enumeration to let $2$-bounded objects be zero (as opposed to for instance `$0$-bounded') is motivated in \cite{DKS} by corresponding periodicity of semi-orthogonal decompositions; we can observe the shadow of this behaviour in Proposition~\ref{CorMinDeg} below. The choice is also convenient for tensor categories, see for instance Theorem~\ref{ThmVer}(1).
\end{enumerate}
\end{remark}

 The following result is a special case of \cite[Proposition~4.1.4(b)]{DKS}, obtained via a completely different proof. We will henceforth say that $X$ is {\bf strictly homotopically $N$-bounded} if it is homotopically $N$-bounded, but not homotopically $N'$-bounded for $N'<N$.
 
 \begin{prop}\label{CorMinDeg}
Assume that $X$ is homotopically bounded, then there exists (a unique) $N\in\mZ_{>1}$ such that $X$ is homotopically $M$-bounded if and only if~$N|M$. 
\end{prop}
 \begin{proof}
 Using \autoref{PropSph}(4) and \autoref{CorTria}(2), if $X$ is homotopically $N$-bounded, then $\mE_{N-2}(X^{(i)})$ is left (and in fact also right, by the second isomorphism in \autoref{CorTria}(2)) invertible, for every $i\in\mZ$. Hence, by \autoref{CorTria}(3)(a) and (b) applied to all $X^{(i)}$, we find that~$X$ is then homotopically $n$-bounded if and only if it is homotopically $n-N$-bounded. The conclusion then follows easily.
 \end{proof}


\section{Pivotal case}\label{SecPiv}

\subsection{Pivotal objects}

We let $\cA$ again be an additive monoidal category.  We have the commutative, see \cite[Proposition~2.2.10]{EGNO}, ring
$$R:=\End_{\cA}(\unit).$$

\subsubsection{}

Following \cite{Sh}, we refer to a pair $(X,\varphi)$ of an object $X\in\cA$ (which admits the first two left adjoints) and an isomorphism
$$\varphi: X\xrightarrow{\sim} X^{\ast\ast}\quad \mbox{in } \;\cA, $$
as a {\bf pivotal object} in $\cA$. The underlying object of a pivotal object is clearly rigid, in the sense of \ref{duals}.

\subsubsection{} As will be made more precise in \ref{unifun}, when considering a pivotal object, we can pretend to be working in a pivotal category. In particular, we have the categorical dimensions as in \cite[Definition~4.7.11]{EGNO}
$$d_X:=\dim X=\ev_{X^\ast}\circ (\varphi\otimes X^\ast)\circ \co_X,\quad d_{X^\ast}:=\dim X^\ast =\ev_{X}\circ (X^\ast\otimes \varphi^{-1} )\circ \co_{X^\ast}$$
in $R$. We will abbreviate
$$\bR(X)\;:=\; (R,d_{X^\ast},d_X),$$
even though the triple depends on $\varphi$. We also have
$$\bR(X)'\;=\;\bR(X^\ast),$$
if we consider the pivotal object $(X^\ast,(\varphi^\ast)^{-1})$.

Our main results in this section are summarised in the following theorems, which will be proved in Section~\ref{SecProofs}. For a pivotal object $(X,\varphi)$ in $\cA$, we have an obvious algebra morphism,
$$\TL_n(\bR(X))\;\to\; \End_{\cA}(\cdots \otimes X\otimes X^\ast\otimes X),$$
where the tensor product has $n$ factors, see \ref{unifun} for more details.
\begin{theorem}\label{ThmPiv1}  The following conditions are equivalent on a triple $\bR:=(R,\delta_1,\delta_2)$ and $n\in\mZ_{>0}$:
\begin{enumerate}
\item $\bR$ is $\cJW_n$;
\item for every pseudo-abelian monoidal category $\cA$, with pivotal object $(X,\varphi)$ with $\bR(X)=\bR$, the complex $\mE_n(X)$ is isomorphic in $\cK(\cA)$ to an object in $\cA\subset\cK(\cA)$.
\end{enumerate}
More explicitly, in $\cK(\cA)$ we have (when $\bR(X)=\bR$ is $\cJW_n$)
$$\mE_n(X)\;\simeq\;\JW_n(\cdots\otimes X\otimes X^\ast\otimes X)[0].$$
\end{theorem}

The combination of Theorem~\ref{ThmPiv1} and Proposition~\ref{PropSph}(3) has the following corollary.
\begin{corollary}\label{ThmPiv2}
Fix $N\in\mZ_{>1}$ and assume that $\bR=\bR(X)$ is $\cJW_{N-1}$ and $\cJW_{N-1}'$. Then~$X$ is homotopically $N$-bounded if and only if 
$$\JW_{N-1}(\cdots\otimes X\otimes X^\ast\otimes X)=0\quad\mbox{and}\quad \JW_{N-1}'(\cdots\otimes X^\ast\otimes X\otimes X^\ast)=0$$
for $\JW_{N-1},\JW_{N-1}'$ the Jones-Wenzl idempotents in $\TL_{N-1}(\bR)$ and $\TL_{N-1}(\bR')$.
\end{corollary}

We can also interpret the rotatability of Jones-Wenzl idempotents:
\begin{theorem}\label{ThmRot}
Consider a triple $\bR=(R,\delta_1,\delta_2)$ with $R$ a field and $N\in\mZ_{>1}$ for which $\bR$ is $\cJW_{N-1}$. Then the following conditions are equivalent on $\bR$:
\begin{enumerate}
\item $\bR$ is $\cJWR_{N-1}$, explicitly
$$[[N]]_{\bR}\;=\;0\;=\; [[N]]_{\bR'}.$$
\item There exists a pseudo-abelian monoidal category $\cA$ with a pivotal object $X\in\cA$, with $\bR(X)=\bR$, that is {\rm strictly} $N$-bounded.
\end{enumerate}
Moreover, under these conditions, a pivotal object $X$ with $\bR(X)=\bR$ is homotopically $N$-bounded if and only if $\JW_{N-1}(\cdots\otimes X\otimes X^\ast\otimes X)=0$.
\end{theorem}

\begin{remark}
In principle, we can improve Theorem~\ref{ThmRot}, since the condition in (1) implies (by definition) $\cJW_{N-1}$. Moreover, we currently have no examples of the situation in (2) with $\bR$ that are not $\cJW_{N-1}$ and $\cJW_{N-1}'$. For $N<6$ one can prove directly that no such examples exist.
\end{remark}

\begin{remark}
The proof of Theorem~\ref{ThmRot}(2)$\Rightarrow$(1) will actually show that if, for some pivotal $X$ and the minimal $N\in\mZ_{>0}$ with $\mE_{N-1}(X)= 0$ in $\cK(\cA)$ (assuming it exists), the condition $\JW_{N-1}$ on $\bR(X)$ implies that also $\mE^{N-1}(X)=0$.
\end{remark}

\begin{remark}
Let $(X,\varphi)$ be a pivotal object with $X$ homotopically $N$-bounded and set $\bR=\bR(X)$. It then follows that
$$[N]_{\bR}\;=\;0\;=\;[N]_{\bR'}.$$
Under the assumptions in Theorem~\ref{ThmRot}, for strictly $N$-bounded $X$ this can be improved to
$$[[N]]_{\bR}\;=\;0\;=\;[[N]]_{\bR'}.$$
\end{remark}

\subsection{Universal pivotal category}\label{SecUni}

Fix a triple $\bR=(R,\delta_1,\delta_2)$.

\subsubsection{} We define a strict monoidal $R$-linear category $2\TL(\bR)$. Objects are words in the alphabet $\{\wedge, \vee\}$. Morphisms from $w$ to $w'$ are $R$-linear combinations of pairings of all letters in $w w'$ such that :
\begin{itemize}
\item We only have pairings between identical letters in different words, or distinct letters in the same word.
\item When placing $w$ on a horizontal line, and $w'$ on a second horizontal line above it, we can connect the two dots in each pair in a way that yields a non-crossing diagram.
\end{itemize}
Multiplication is defined by concatenation of diagrams, with evaluation of counter-clockwise loops at $\delta_1$ and clockwise loops at $\delta_2$. The tensor product is given by juxtaposition of diagrams. In particular, the empty word $\varnothing$ is the tensor unit. We have a canonical isomorphism
\begin{equation}\label{TLEnd}
\TL_n(\bR)\;\xrightarrow{\sim}\; \End_{2\TL(\bR)}(\cdots\wedge\vee\wedge\vee\wedge),
\end{equation}
whereas
$$\End_{2\TL(\bR)}(\wedge\wedge\cdots\wedge)\;\simeq\;R.$$
Clearly ${2\TL(\bR)}$ is rigid (and pivotal), with $\wedge^\ast=\vee={}^\ast \wedge$.

\subsubsection{}\label{unifun}Now let $\cA$ be an additive monoidal category with a pivotal object $(X,\varphi)$ and set $\bR:=\bR(X)$. There exists an $R$-linear monoidal functor
$$2\TL(\bR)\;\to\; \cA,$$
which satisfies $\wedge\mapsto X$ and $\vee\mapsto X^\ast$. On morphisms the functor is determined as follows. The clockwise cup is sent to $\co_X$, the counter-clockwise cap to $\ev_X$, the counter-clockwise cup to the composite
$$\unit\xrightarrow{\co_{X^\ast}} X^\ast\otimes X^{\ast\ast}\xrightarrow{X^\ast\otimes \varphi^{-1}}X^\ast\otimes X,$$
and the clockwise cap to the composite
$$X\otimes X^\ast\xrightarrow{\varphi\otimes X^\ast}X^{\ast\ast}\otimes X^\ast\xrightarrow{\ev_{X^\ast}}\unit.$$

We let $2\TL(\bR)^{en}$ be the pseudo-abelian envelope of $2\TL(\bR)$. It comes with a similar universal property amongst pseudo-abelian monoidal categories.

\begin{question}
For each rotatable Jones-Wenzl idempotent, we can consider the tensor ideal in $2\TL(\bR)$ it generates.
In case $R$ is a field, are these ideals prime? Are there any other prime ideals? 
Note that there do exist other tensor ideals in $2\TL(\bR)$ which are not prime, for instance the ideal generated by $\id_{\wedge\wedge}$.
\end{question}

\subsection{Proofs}
\label{SecProofs}
\begin{proof}[Proof of Theorem~\ref{ThmPiv1}]
It is sufficient to prove that $\JW_n$ exists in $\TL_n(\bR)$ if and only if~$\mE_n(\wedge)$, for $\wedge\in 2\TL(\bR)^{en}$, is homotopy equivalent to a chain complex in degree 0.

Any such homotopy equivalence for any chain complex of the form
$$0\to C_0\xrightarrow{d_0}C_{-1}\xrightarrow{d_{-1}}C_{-2}\to\cdots$$
clearly implies the existence of an idempotent $e\in\End(C_0)$ with $d_0e=0$ and such that $\id_{C_0}-e$ factors through $d_0$. In our case, the latter vanishing condition is equivalent to demanding $e=\JW_n\in\TL_n(\bR)$, see Definition~\ref{DefJW}. Hence (2) implies (1). This also already proves that if (2) is satisfied, $\mE_n(X)$ is isomorphic to the evaluation of $\JW_n$ as claimed in the last sentence of the theorem.

Now we prove the other direction (1)$\Rightarrow$(2). We thus assume that $\JW_n$ exists in $\TL_n(\bR)$ and consider $\cA=2\TL(\bR)^{en}$ and $X=\wedge$. We use the description of $\mE_n(X)$ in Remark~\ref{RemExp}
$$0\to A\to \bigoplus_{|I|=2}A_I\to\bigoplus_{|J|=4}A_J\to\cdots.$$
We can write
$$\JW_n\;\in\; \TL_n(\bR)\;\simeq\; \End_{2\TL(\bR)}(A)$$
non-uniquely, and with the notation from Remark~\ref{RemExp}, as 
$$\JW_n\;=\; \id_A-\sum_{|I|=2} \eta^{\varnothing}_I\circ \ev_I^{\varnothing} $$
for chosen morphisms $\eta^{\varnothing}_I: A_I\to A$. The morphism $\oplus_{|I|=2}A_I\to A$, resulting from $\{\eta^\varnothing_I\}$, is the first morphism in the chain homotopy we need to construct, which we can do iteratively.

Consider therefore the following diagram, for $l\ge 1$,
$$\xymatrix{
\bigoplus_{|I|=2l-2}A_I\ar[rr]\ar@{=}[d]&&\bigoplus_{|J|=2l}A_J\ar@{=}[d]\ar[dll]^{\eta}\ar[rr]&&\bigoplus_{|L|=2l+2}A_L\ar@{=}[d]\\
\bigoplus_{|I|=2l-2}A_I\ar[rr]&&\bigoplus_{|J|=2l}A_J\ar[rr]&&\bigoplus_{|L|=2l+2}A_L,
}$$
with an appropriate matrix $\eta_J^I$ constructed in the previous step. Here, `appropriate' means in particular that the two paths from $\oplus A_I$ on the top row to $\oplus A_J$ on the lower row compose to identical maps. We need to show that, for every two $J,J'$ (with $|J|=2l=|J'|$), the map
$$\alpha_{J',J}\;:=\; \delta_{J',J}\,\id_{A_J}-\sum_{I\lhd J'}\ev^{I}_{J'}\circ \eta_{J}^I\;:\; A_J\to A_{J'},$$ factors via $\oplus A_L$. By assumption, for any fixed $I'$, we have
\begin{equation}\label{suma1}
\sum_{J\rhd I'} \alpha_{J',J}\circ \ev^{I'}_J\;=\;0,
\end{equation}
as morphisms $A_{I'}\to A_{J'}$. We thus choose one $ J\rhd I'$, define $i$ by $J\backslash I'=\{i,i+1\}$, and rewrite \eqref{suma1} as
\begin{equation}\label{suma2}
\alpha_{J',J}\circ \ev^{I'}_J\;=\;-\sum_{J''\not=J} \alpha_{J',J''}\circ \ev^{I'}_{J''}.
\end{equation}
By diagrammatic considerations, we can immediately reduce the summation (while keeping the equality) in the right-hand side  of~\eqref{suma2} to those $J''\rhd I'$ that are disjoint from the pair $\{i,i+1\}$. It thus follows that all diagrams appearing (with non-zero coefficient) in the right-hand side contain a cap in some position $\{t,t+1\}$ (with $J''=I'\sqcup\{t,t+1\}$) as well as in position $\{i,i+1\}$. We can thus restructure~\eqref{suma2} as saying
$$\alpha_{J',J}\circ \ev_{J}^{I'}\;=\; \sum_{L\rhd J}\beta_L\circ \ev^J_L\circ \ev^{I'}_J$$ 
for some morphisms $\beta_L: A_L\to A_{J'}.$ Clearly every whiskered evaluation is an epimorphism in $2\TL(\bR)^{en}$, so we find the desired factorisation of $\alpha_{J',J}$.
This iterative argument shows that the chain homotopy exists.
\end{proof}

\begin{proof}[Proof of Theorem~\ref{ThmRot}]
We start by observing that under condition (1), homotopical $N$-boundedness is characterised as in the last phrase of the theorem. Indeed, by construction and the defining property of rotatability in \cite[\S 2]{Ha}, both conditions in~\autoref{ThmPiv2} become equivalent under assumption (1).

In this paragraph we let $R$ be an arbitrary commutative ring, not necessarily a field.
We have a morphism of $R$-modules from \cite[Definition~6.13]{EW}
$$p_n:\TL_n(\bR)\;\to\; \TL_{n-1}(\bR).$$
For example, when $n$ is even, in the categorical interpretation, see \eqref{TLEnd}, with $X$ for $\wedge$, it is given by
$$p_n:\;f\;\mapsto \; (\ev_{X^\ast}\otimes X\otimes X^\ast\otimes\cdots\otimes X)\circ (X\otimes f) \circ (\co_X\otimes X\otimes X^\ast\otimes\cdots\otimes X).$$ 
It is proved in \cite[(6.27)]{EW} that in case $\JW_{i}$ exist in $\TL_i(\bR)$ for $M-1\le i<N$ (in particular $[M]_R$ is invertible), then
$$p_{M}\circ p_{M-1}\circ\cdots \circ p_{N-1}(\JW_{N-1})\;=\; \frac{[N]_{\bR}}{[M]_{\bR}}\,\JW_{M-1}.$$
This implies that
\begin{equation}\label{Eqpp}
\left(\prod_{i}[[i]]_{\bR}\right)\;p_{M}\circ p_{M-1}\circ\cdots \circ p_{N-1}(\JW_{N-1})\;=\; \left(\prod_{l}[[l]]_{\bR}\right)\;\JW_{M-1},\end{equation}
where $i$ runs over all divisors of~$M$ that do not divide $N$ and $l$ runs over all divisors of $N$ that do not divide $M$. 
Now we only assume that $\JW_{N-1}$ exists. Whenever the coefficient of $\JW_{M-1}$ in \eqref{Eqpp} is invertible in $R$, it follows from the genericity of the coefficients of Jones-Wenzl idempotents, see \cite[Theorem~4.4]{Ha}, that $\JW_{M-1}$ exists and satisfies \eqref{Eqpp}.

Now let $R$ again be a field. First assume that (2) is satisfied, so we have corresponding $\cA,X$. In order to find a contradiction, assume that (1) is not satisfied. By Proposition~\ref{PropRot0}(3), there exists $1\le M< N$ for which the scalar factor in the right-hand side of~\eqref{Eqpp} is not zero and hence invertible.
 Equation~\eqref{Eqpp} then implies that $\JW_{M-1}$ exists and must act as zero on $\cdots X^\ast\otimes X$ (so $\mE_{M-1}(X)=0$ by Theorem~\ref{ThmPiv1}). In case~$\bR$ is (also) not~$\cJWR_{M-1}$, we can repeat the above procedure to arrive at yet another $M'<M$ for which $\JW_{M'-1}$ exists and $\mE_{M'-1}(X)=0$. By iteration this procedure must thus stop at some $M'<N$ for which $\bR$ is $\cJWR_{M'-1}$ (and $\mE_{M'-1}(X)=0$), for instance because $\bR$ is always $\cJWR_{1}$. By the first paragraph of the proof, $X$ is homotopically $M'$-bounded, contrary to assumption.

Now assume condition (1) is satisfied. We denote the tensor ideal in $2\TL(\bR)^{en}$ generated by $\JW_{N-1}$ by $J$. By the first paragraph of the proof, the generator $X=\wedge$ is homotopically $N$-bounded in $2\TL(\bR)^{en}/J$. Moreover, we claim that 
\begin{equation}\label{ttois}
\Hom_{2\TL(\bR)}(Y,Z)\;\tto\;\Hom_{2\TL(\bR)/J}(Y,Z)
\end{equation}
is an isomorphism if $Y$ and $Z$ are of the form $\cdots\otimes X\otimes X^\ast\otimes X$ where the sum of the numbers of factors in $Y$ and $Z$ is strictly below $2N-2$.
Indeed, this follows easily from the diagrammatics, and the fact that $p_{N-1}(\JW_{N-1})=0$ under condition (1), see \cite[\S 2]{Ha}. In particular, $X\in 2\TL(\bR)^{en}/J$ is not homotopically $M$-bounded for $M<N$. Indeed, the latter would imply that the identity of $\cdots\otimes X^\ast\otimes X$  (with $M-1$ factors in the product) factors through the corresponding product with $M-3$ factors, contradicting isomorphisms~\eqref{ttois}.
\end{proof}


\section{Tensor categories: general theory}\label{SecTens1}

In the abelian setting, we can also apply the technology of \cite{DKS} to the derived category, rather than the homotopy category of chain complexes. For brevity, we move immediately to the case of tensor categories over a field here. 

Henceforth, unless further specified, we let $\cC$ denote an arbitrary tensor category over a field $\bk$. We let $\Inj\cC$ denote the category of injective objects in $\Ind\cC$.

\begin{definition}
An object $X\in\cC$ is {\bf $N$-bounded}, with $N\in\mZ_{>0}$, if the endofunctor $-\otimes X$ of the stable $\infty$-category $\mathrm{N}_{\mathrm{dg}}(\mathrm{Ch}^+(\Inj\cC))$ is $N$-spherical in the sense of \cite[Definition~4.1.1]{DKS}.
\end{definition}

We will again use the obvious terminology of {\bf (un)bounded} and {\bf strictly} $N$-bounded objects.

\subsection{Bounded objects}

We state some properties which are immediate analogues of the corresponding observations in Section~\ref{SecAdd}. In particular, boundedness is now determined by complexes being acyclic. The following proposition will be improved in Theorem~\ref{ThmK0}.

\begin{prop}\label{PropSph2}For $N\in \mZ_{>1}$, the following conditions on $X\in\cC$ are equivalent:
\begin{enumerate}
\item $X$ is $N$-bounded;
\item 
$\mE_{N-1}(X)\simeq 0\simeq \mE^{N-1}(X)\;$ in $\;\cD(\cC)$;
\item $\mE_{N-1}(X)\simeq 0\simeq \mE_{N-1}(X^\ast)\;$ in $\;\cD(\cC)$;
\item $\mE_{N-1}(X^{(i)})\simeq 0 \;$ in $\;\cD(\cC)\;$ for all $i\in\mZ.$
\end{enumerate}
\end{prop}

\begin{example}
\begin{enumerate}
\item No objects are 1-bounded.
\item The 2-bounded objects are precisely the zero objects.
\item The 3-bounded objects are precisely the invertible objects.
\end{enumerate}
\end{example}

Invertible objects in $\cD(\cC)$ correspond to invertible objects in $\cC$ placed in an arbitrary homological degree, and left invertible objects are right invertible. Therefore, the analogue of \autoref{CorTria} can be formulated as follows.

\begin{corollary}\label{CorTria2}
If $\mE_{N-1}(X)= 0$ in $\cD(\cC)$ for $X\in\cC$ and $N\in\mZ_{>1}$, then in $\cD(\cC)$:
\begin{enumerate}
\item $\mE_{N-2}(X)\simeq\mE_N(X)[1]$;
\item For $i\in\{N,N-2\}$ and $j\in\mZ$, the objects $\mE_{i}(X^{(j)})$ are invertible;
\item For $m\in\mN$ and $j\in\mZ$ with $m+j$ odd, the objects
$\mE_{m+N}(X^{(j)})[1]$ and $\mE_{m}(X^{(j)})$ are isomorphic, up to tensor product with an invertible object;
\item For $m\in\mN$, the objects $\mE_{m+N}(X)[1]$ and $\mE_{m}(X^{(N)})$ are isomorphic, up to tensor product with an invertible object.
\end{enumerate}
\end{corollary}

 \begin{corollary}\label{CorNM}
Assume that $X$ is bounded, then there exists (a unique) $N$ such that $X$ is $M$-bounded if and only if $N|M$.
\end{corollary}

Objects can be (strictly) $N$-bounded for any $N>1$, already in fusion categories over $\bk=\mC$:
\begin{example}\label{exallN}
\begin{enumerate}

\item Let $q=\zeta_{2N}$ be a primitive $2N$-th root of unity, for $N>2$. The image of the natural $U_q(\mathfrak{sl}_2)$-representation in the Verlinde category $\cC(\mathfrak{sl}_2,q)$, see \cite[Example~8.18.5]{EGNO}, is strictly $N$-bounded. This follows for instance from Theorem~\ref{ThmFP} below, or from Theorem~\ref{ThmPiv1}.
\item As a special case of part (1), the generator of the Ising fusion category is strictly 4-bounded, see \cite[Proposition~B.3]{DGNO}.
\item The generator $X$ of the Tambara-Yamagami fusion category for $A=\mZ/3$, see \cite{TY}, is strictly 6-bounded. Since $X^{\otimes 2}\simeq \unit^3$, this follows from Theorem~\ref{ThmFP}.
\end{enumerate}
\end{example}

By the K\"unneth formula, it follows that when $A\otimes B=0$ for $A,B\in\cD^b(\cC)$, then $A=0$ or $B=0$. In particular, Proposition~\ref{PropDKS}(1) implies the following lemma.
\begin{lemma}\label{InvZero}
Consider $X\in\cC$ and $i\in\mZ_{>0}$. If $\mE_i(X)$ is invertible, then $\mE_{i-1}(X)=0$ or $\mE_{i+1}(X)=0$.
\end{lemma}

Also the following observation is immediate, as tensor functors are faithful and exact.
\begin{lemma}\label{TensorF}
Consider a tensor functor $F:\cC\to\cC'$ between tensor categories and an object $X\in\cC$. Then $X$ is $N$-bounded if and only if $F(X)$ is $N$-bounded.
\end{lemma}

\begin{remark}
If $X$ is pivotal so that $\bR(X)$ is $\cJW_{N-1}$ and $\cJW_{N-1}'$, then  (by Theorem~\ref{ThmPiv1}) $X$ is $N$-bounded if and only if it is homotopically $N$-bounded. In general, there are $N$-bounded objects in tensor categories which are not homotopically $N$-bounded, see \autoref{PropYang} below.
\end{remark}

\subsection{Alternative characterisations}

A remarkable feature of boundedness in tensor categories is that it can be detected in the Grothendieck ring, and by only one of the continuant complexes.

\begin{theorem}\label{ThmK0}
The following properties are equivalent on $X\in\cC$ and $N\in\mZ_{>1}$:
\begin{enumerate}
\item $X$ is $N$-bounded;
\item $\mE_{N-1}(X)=0$ in $\cD(\cC)$;
\item $[\mE_{N-1}(X)]=0$ in $K_0(\cC)$;
\item
\begin{enumerate}
\item $\mE_{N-2}(X)$ and $\mE_N(X)$ are invertible;
\item If $X$ is strictly $4$-bounded, then $N$ is a multiple of $4$.
\end{enumerate}
\end{enumerate}
\end{theorem}

The remainder of this section is devoted to the proof of the theorem.
\begin{remark}\label{RemAst}
\begin{enumerate}
\item By symmetry, we can add $\mE^{N-1}(X)=0$ and $[\mE^{N-1}(X)]=0$ to the list of equivalent conditions in Theorem~\ref{ThmK0}.
\item The proofs (or more concretely, the combination of Theorems~\ref{ThmK0} and~\ref{ThmConjP} and Corollary~\ref{CorTria2}(3)) actually show that the homology of $\mE_{n}(X)$ is always contained in one degree, explaining equivalence between \ref{ThmK0}(2) and~(3).
\end{enumerate}

\end{remark}
 For convenience we label the following objects in $\cC$, for $i\in\mN$:
$$\mV_i(X)\,:=\, H_0(\mE_i(X)),\quad \mV^i(X)\,:=\,H_0(\mE^i(X))\,\simeq\,\mV_i({}^\ast X)^\ast.$$

\begin{theorem}\label{ThmConjP}
For every $X\in \cC$, there exists $n\in\mZ_{>0}\cup\{\infty\}$ so that:
\begin{enumerate}
\item For all $0\le i<n$, we have $\mE_i(X)\simeq \mV_i(X)[0]\not=0$ in $\cD(\cC)$;
\item We have $\mE_n(X)= 0$ in $\cD(\cC)$ (assuming $n<\infty$);
\item We have $\mV_i(X)=0$ for all $i\ge n$.
\end{enumerate}
\end{theorem}
\begin{proof}
By definition $\mE_0(X)=\unit[0]$ and $\mE_1(X)=X[0]$, hence parts (1) and (2) follow by iteration from the following observation:
 Assume that $\mE_m(X)$ and $\mE_{m-1}(X)$ are quasi-isomorphic to $A[0]$ and $B[0]$ for non-zero $A,B\in\cC$; then $\mE_{m+1}(X)$ is also quasi-isomorphic to a complex concentrated in degree $0$. This observation is a consequence of Proposition~\ref{PropDKS}(1). Concretely, the latter implies that we have an exact triangle in $\cD(\cC)$:
 $$B^\ast \otimes\mE_{m+1}(X)\to A^\ast \otimes A [0]\xrightarrow{\ev_A}\unit[0]\to.$$
 Hence,  $B^\ast \otimes\mE_{m+1}(X)$ is quasi-isomorphic to the kernel of $\ev_A$, placed in degree 0.
 
 Part (3) follows from the observation that, since the complexes $\mE_i(X)$ are by construction contained in non-positive homological degrees, equation~\eqref{DefTria} yields an inclusion
 $$H_0(\mE_{i}(X))\;\hookrightarrow\; X^{(i-1)}\otimes H_0(\mE_{i-1}(X)).$$
In particular, $\mV_i(X)=0$ implies $\mV_{i+1}(X)=0$, concluding the proof. 
\end{proof}


\begin{proof}[Proof of Theorem~\ref{ThmK0}]

Let $N_0$ be the minimal $N\in\mN$ (assuming it exists) for which $\mE_{N-1}(X)=0$ in $\cD(\cC)$. It follows from Theorem~\ref{ThmConjP} that this is also the minimal $N$ for which $[\mE_{N-1}(X)]=0$. By Corollary~\ref{CorTria2}(2), it follows that $\mE_{N_0-2}(X^{(j)})$ is invertible for all $j\in\mZ$. From Lemma~\ref{InvZero}, applied to $X^\ast$, it then follows that either $\mE^{N_0-1}(X)=0$ or $\mE^{N_0-3}(X)=0$. If the latter would be the case, we can repeat the argument to find that either $\mE_{N_0-3}(X)=0$ or $\mE_{N_0-5}(X)=0$, contradicting minimality of $N_0$. Hence, $\mE^{N_0-1}(X)=0$. Furthermore, by symmetry, or the above reasoning, $N_0$ is also the minimal $N$ for which $\mE^{N-1}(X)=0$ and also the minimal $N$ for which $[\mE^{N-1}(X)]=0$.

Subsequently, it follows from Corollary~\ref{CorTria2}(3) that $\mE_n(X^{(j)})[1]$ and $\mE_{n-N_0}(X^{(j)})$ are isomorphic in $\cD(\cC)$ up to tensor product with an invertible object, for all $n\ge N_0$ and $j\in\mZ$. Hence $[\mE_n(X^{(j)})]$ and $[\mE_{n-N_0}(X^{(j)})]$ are also equal in $K_0(\cC)$, up to multiplication with a unit. It thus follows easily that all $N$ for which one (or all) of $\mE_{N-1}(X)$, $\mE^{N-1}(X)$, $[\mE_{N-1}(X)]$, $[\mE^{N-1}(X)]$ are zero are simply the multiples of $N_0$. This implies equivalence of (1), (2) and~(3).

Next, we prove that (1) implies (4). 
 Condition (4)(a) is already established in Corollary~\ref{CorTria2}(2). 
Condition 4(b) follows Corollary~\ref{CorNM}.

Finally, we prove that (4) implies (1).
 The conditions in (4)(a) imply, by Lemma~\ref{InvZero}, that either $\mE_{N-1}(X)=0$ or alternatively $\mE_{N-3}(X)=0=\mE_{N+1}(X)$. As we already know that (2) implies (1), we find that $X$ is either $N$-bounded, or $N\pm 2$-bounded. If the former is true we are done, so assume that $X$ is $N\pm 2$-bounded. We can now conclude by Corollary~\ref{CorNM}. Indeed, $X$ is strictly $2$-bounded or $4$-bounded. However, $X$ is not 4-bounded by condition~(4)(b). So $X=0$, which implies $\mE_{N-1}(X)=0$ as desired. 
\end{proof}

\subsection{Simplicity}

\begin{theorem}\label{ThmSimp}
If $X\in\cC$ is bounded, then $X$ is simple or zero.
\end{theorem}
\begin{proof}
Let $X$ be strictly $N$-bounded, for $N>2$. Set $\mV_k=\mV_k(X)$.
 By Theorems~\ref{ThmK0} and~\ref{ThmConjP}, for $k<N$ we have $\mV_k[0]\simeq \mE_k(X)$, and $\mV_k=0$ if and only if $k=N-1$. Assume for a contradiction that $X$ is not simple. Then triangle~\eqref{DefTria} implies that
 $$\ell(\mV_k)+\ell(\mV_{k-2})\;\ge\; 2\ell(\mV_{k-1}),$$
 for all $k<N$. By iteration on $i\ge 0$, we then find that
 $$\ell(\mV_{N-1})-\ell(\mV_{N-2})\;\ge\; \ell(\mV_{N-1-i})-\ell(\mV_{N-2-i}).$$
For $i=N-2$, we find, using $\mV_0=\mE_0(X)=\unit$,
$$-\ell (\mV_{N-2})\;\ge\; \ell(\mV_1)-1.$$
The left-hand side is strictly negative, while the right-hand side is positive (strictly, in fact, as $\mV_1=X$), a contradiction. 
\end{proof}

\begin{remark}
For quasi-finite or pivotal tensor categories, Theorem~\ref{ThmSimp} also follows from Theorems~\ref{ThmFP} and~\ref{ThmDX} below.
\end{remark}

\begin{corollary}\label{LemMustPiv}
Let $X\in\cC$ be homotopically $N$-bounded for $N$ odd, then $X^{\ast\ast}\simeq X$. 
\end{corollary}
\begin{proof}
The right-most term of $\mE_{N-1}(X)$ is $\unit$. The required chain homotopy requires a non-zero morphism
$$\unit\to X^{(i+1)}\otimes X^{(i)}$$
for some $0\le i<N-1$. As $X$ is simple by Theorem~\ref{ThmSimp}, this requires $X^{(i+1)}\simeq X^{(i-1)}$, and consequently, $X^{\ast\ast}\simeq X$.
\end{proof}


\section{Tensor categories: specific instances}\label{SecTens2}

From now on we assume that $\bk$ is algebraically closed. Unless further specified, $\cC$ is an arbitrary tensor category over $\bk$.

\subsection{Quasi-finite tensor categories}

\begin{theorem}\label{ThmFP}Let $\cC$ be a quasi-finite tensor category, $X\in\cC$ and $N\in\mZ_{>0}$.
\begin{enumerate}
\item $X$ is strictly $N$-bounded if and only if
$$\FPdim X\;=\; 2\cos\left({\frac{\pi}{N}}\right).$$
\item $X$ is unbounded if and only if $\FPdim X\ge2$.
\end{enumerate}
\end{theorem}
The remainder of this section is devoted to the proof of the theorem.

\begin{remark}
 That Frobenius-Perron dimensions must be in the `Jones spectrum'
$$\{2\cos\left(\frac{\pi}{N}\right)\,|\, N\ge 2\}\;\cup\;\mR_{\ge 2}$$
is well-known, see \cite[Corollary~3.3.16]{EGNO}, and we could simplify the proof below by using this fact. We avoid this, so that the proof will also work in the setting of Section~\ref{SubSecPiv}.

The numbers $2\cos\left({\frac{\pi}{N}}\right)$ can be characterised as the only algebraic integers in the real interval $[0,2[$ which are totally real and of maximal absolute value in their Galois orbit.
\end{remark}

\begin{lemma}\label{Lem4Dim}
Let $\phi:K_0(\cC)\to R$ be a ring homomorphism to a commutative ring $R$ such that $\phi([X^{(i)}])=\phi([X])$ for all $i\in\mZ$. Then,
$$K_0(\cD^b(\cC))\;\xrightarrow{\sim}\;K_0(\cC)\;\xrightarrow{\phi}\;R\quad\mbox{satisfies}\quad [\mE_{n}(X)]\mapsto \kappa_n(\phi([X])).$$ In particular, if $X\in\cC$ is $N$-bounded then
$$\kappa_{N-1}(\phi([X]))\;=\;0.$$
\end{lemma}
\begin{proof}
This follows from comparing \eqref{DefTria} with the definition in \ref{DefCheb}, or from Remark~\ref{RemCont}.
\end{proof}

\begin{example}\label{Expsi2}
If $\cC$ is a quasi-finite symmetric tensor category and $X\in\cC$ is $N$-bounded, then
$$\FPdim(\Sym^2X)-\FPdim(\wedge^2 X)\;=\; 2\cos\left(\frac{j\pi}{N}\right)$$
for some $1\le j<N$. Indeed, this follows from Lemma~\ref{Lem4}(2), and Lemma~\ref{Lem4Dim} applied to the ring homomorphism obtained by composing the second Adams operation, see \cite[Lemma~4.4]{EOV}, with the Frobenius-Perron dimension
$$K_0(\cC)\xrightarrow{\psi^2} K_0(\cC)\xrightarrow{\FPdim}\mR.$$ Note that $\mathrm{char}(\bk)=2$ is a special case. Then one should replace $\psi^2$ by $[X]\mapsto [\Fr X]$.
It will actually follow from (the proof of) Theorem~\ref{Thmpn} that the only values which occur are for $j=2$ or $j=N-2$.
\end{example}

%

\begin{proof}[Proof of Theorem~\ref{ThmFP}]
We use Theorem~\ref{ThmK0} freely.
Fix $N>1$ and some $X\in\cC$ for which
\begin{equation}\label{eq2ass}\FPdim X \;>\; 2\cos{\frac{\pi}{N-1}},\quad\mbox{and}\quad\mE_m(X)\not=0\;\mbox{ for all $m<N-1$.} \end{equation}
It follows from the second assumption in \eqref{eq2ass} and Theorem~\ref{ThmConjP} that $\mE_{N-1}(X)$ is quasi-isomorphic to $\mV_{N-1}(X)[0]$. Then, by Lemma~\ref{Lem4Dim},
$$\FPdim\mV_{N-1}(X)= \kappa_{N-1}(\FPdim X)=\prod_{1\le j<N}\left(\FPdim X-2\cos\left(\frac{j\pi}{N}\right)\right).$$
This value must be non-negative, but by the first assumption in \eqref{eq2ass}, every factor in the right-hand side, except potentially the first one, must be strictly positive. 
Hence there are two options: either $\FPdim X=2\cos\left(\frac{\pi}{N}\right)$ (so that the first factor is zero) and $X$ is strictly $N$-bounded; or alternatively $\FPdim X>2\cos\left(\frac{\pi}{N}\right)$ and $\mE_m(X)$ is non-zero in $\cD(\cC)$ for all $m<N$ (our two starting assumptions in \eqref{eq2ass}, but now for $N+1$).

In case $N=2$, assumptions \eqref{eq2ass} are automatic. Part (1) thus follows from the above iterative procedure.

The procedure also implies that for $X$ to be unbounded, we must have
$$\FPdim X\;>\; 2\cos\left(\frac{\pi}{N}\right)$$
for all $N>1$, from which part (2) follows.
\end{proof}

The following corollary is immediate from Theorem~\ref{ThmFP}.
\begin{corollary}\label{CorFP}
Let $\cC$ be a quasi-finite tensor category and consider bounded $X,Y\in\cC$. Then the length of $X\otimes Y$ is at most three.

\end{corollary}

\begin{remark}\label{RemCorFP}
Corollary~\ref{CorFP} is sharp (already for symmetric tensor categories, see Theorem~\ref{ThmVer}):
\begin{enumerate}
\item For $p>2$, the $p^{2}$-bounded object $L_1$ and $p$-bounded object $L_p$ in $\Ver_{p^2}$ have tensor product equal to the simple object $L_{p+1}$.
\item For $p>3$, the $p$-bounded object $L_1$ in $\Ver_p$, satisfies $L_1^{\otimes 2}=\unit\oplus L_2$.
\item For $p=2$ and $n>2$, the $2^n$-bounded generator $L_1$ of $\Ver_{2^n}$ is such that $L_1^{\otimes 2}$ has length three; with top and socle $\unit$ and $2^{n-1}$-bounded middle $L_2$, see \cite{BE1}. Another example is the generator of the Tambara-Yamagami category for $A=\mZ/3$, see Example~\ref{exallN}(3).
\end{enumerate}
\end{remark}

\begin{example}\label{ExFP}
The Frobenius-Perron dimension of a strictly $N$-bounded object is
$$
\begin{tabular}{ l | r  }
  N & $2\mathrm{cos}(\pi/N)$  \\
  \hline
  \hline
  2 & 0 \\
  \hline
  3 & 1  \\
  \hline
  4 & $\sqrt{2}$  \\
  \hline
  5 & $(1+\sqrt{5})/2$  \\
  \hline
  6 & $\sqrt{3}$  \\
\end{tabular}
$$
\end{example}

\subsection{Pivotal objects}
\label{SubSecPiv}
For pivotal objects in tensor categories, we can replace the Frobenius-Perron dimension in the analysis of the previous section by the growth dimension in \eqref{DefDX}.


\begin{theorem}\label{ThmDX}Consider $X\in\cC$ with $X^{\ast\ast}\simeq X$ and $N\in\mZ_{>0}$.
\begin{enumerate}
\item $X$ is strictly $N$-bounded if and only if
$$\bD(X)\;=\; 2\cos\left({\frac{\pi}{N}}\right).$$
\item $X$ is unbounded if and only if $\bD(X)\ge2$.
\end{enumerate}
\end{theorem}
\begin{proof}
For $n$ even, we find that $[\mE_{n}(X)]$  equals $\kappa_{n}(\sqrt{[X^\ast][X]})$. For $n$ odd, we find that $[X^\ast\otimes\mE_{n}(X)]$  equals $\sqrt{[X^\ast][X]}\kappa_{n}(\sqrt{[X^\ast][X]})$.
We can then apply Theorem~\ref{ThmConjP} and Lemma~\ref{Lemgd}(2) to $Y:=X^\ast\otimes X$ and the polynomials $g(x)=\kappa_{n}(\sqrt{x})$ or $g(x)=\sqrt{x}\kappa_{n}(\sqrt{x})$, to copy the proof of Theorem~\ref{ThmFP} almost verbatim.
\end{proof}

\begin{remark}
 The observation that $\gd(X^\ast\otimes X)$ must be in the Jones spectrum
$$\{4\cos^2\left(\frac{\pi}{N}\right)\,|\, N\ge 2\}\;\cup\;\mR_{\ge 4}$$
in the generality of Theorem~\ref{ThmDX}, appears to be new.
\end{remark}

\subsection{Symmetric tensor categories}

\begin{theorem}\label{ThmSymm1}
Let $\cC$ be a symmetric tensor category over $\bk$.
\begin{enumerate}
\item If $\mathrm{char}(\bk)=0$, then the only bounded objects in $\cC$ are zero or invertible. 
\item If $\mathrm{char}(\bk)=p>0$ and $X$ is $p^n$-bounded, for $n\in\mZ_{>0}$, then $\dim X=\pm 2$ and $X$ is homotopically $p^n$-bounded.
\end{enumerate}
\end{theorem}
\begin{proof}We start by proving part (1). Let $X\in \cC$ be bounded.
By Theorem~\ref{ThmDX}, the tensor subcategory generated by $X$ is of moderate growth and therefore admits a tensor functor to $\sVec$, see \cite{Del02}. The conclusion thus follows easily from Lemma~\ref{TensorF} and Theorem~\ref{ThmSimp}.

The dimension in part (2)  follows from Lemma~\ref{Lem4Dim} applied to the categorical dimension
$$\dim: K_0(\cC)\;\to\; \mF_p$$ and Lemma~\ref{Lem4}(1).
That the objects must be homotopically bounded follows from the observation that $\mE_{p^n-1}(X)$ and $\mE^{p^n-1}(X)$ are homotopy equivalent to complexes contained in one degree, by Theorem~\ref{ThmPiv1} and Example~\ref{ExEx}(1).
\end{proof}

For the remainder of this section, we assume $\mathrm{char}(\bk)=p>0$.
The following conjecture will be proved for quasi-finite symmetric tensor categories in Section~\ref{SecPfConj}.
\begin{conjecture}\label{Conjpn}
For $\cC$ a symmetric tensor category over $\bk$, consider a strictly $N$-bounded $X\in \cC$  for $N>3$.
\begin{enumerate}
\item Either $\Sym^2 X$ or $\wedge^2 X$ is invertible;
\item $N=p^n$ for some $n\in\mZ_{>0}$;
\item $\dim X=\pm 2$;
\item $X$ is homotopically $N$-bounded.
\end{enumerate}
\end{conjecture}

\begin{remark}
\begin{enumerate}
\item Theorem~\ref{ThmSymm1}(2) already shows that \ref{Conjpn}(2) implies (3) and (4). We will see moreover in Lemma~\ref{LemBEO} that (1) implies (2), so we can focus on \ref{Conjpn}(1).
\item Theorem~\ref{ThmVer} will show that Conjecture~\ref{Conjpn} would follow from \cite[Conjecture~1.4]{BEO}. The validity of \cite[Conjecture~1.4]{BEO} would also imply \cite[Conjecture~8.1.7]{CEO3} that $\gd$ is always a ring homorphism for symmetric tensor categories of moderate growth, see \cite[Lemmata 8.3 and~8.5]{CEO1}. 
 If we would assume \cite[Conjecture~8.1.7]{CEO3}, we can replace $\FPdim$ by $\gd$ in the proofs below and also prove Conjecture~\ref{Conjpn}. We thus get a hierarchy
 $$\mbox{[BEO, Conjecture~1.4]}\;\Rightarrow\; \mbox{[CEO3, Conjecture~8.1.7]}\;\Rightarrow\; \mbox{Conjecture~\ref{Conjpn}}$$
  We can therefore view Theorem~\ref{Thmpn} as evidence towards \cite[Conjecture~1.4]{BEO}.
\item When $\cC$ is Frobenius exact (of moderate growth), we know that $\gd$ is a ring homomorphism by \cite{CEO1}. However, the latter also shows directly (by Theorem~\ref{ThmVer} and Lemma~\ref{TensorF}) that in Frobenius exact tensor categories, bounded objects are either zero, invertible or strictly $p$-bounded, and Conjecture~\ref{Conjpn} is satisfied.
\item Conjecture~\ref{Conjpn}(1) would yield, thanks to \cite[Theorem~4.79]{BEO}, a complete classification of the symmetric tensor categories generated by a bounded object. In particular, we have such a classification for quasi-finite symmetric tensor categories, by Theorem~\ref{Thmpn}.
\end{enumerate}
\end{remark}

\begin{theorem}\label{ThmVer}
Consider $\cC=\Ver_{p^n}$.
\begin{enumerate}
\item The generator $L_1$ is strictly $p^n$-bounded.
\item For $0\le i<n$, the objects $L_{p^i}$ and $L_{p^i+(p-2)p^{n-1}}$ (to be interpreted as $L_{(p-3)p^{n-1}}$ for $i=n-1$) are strictly $p^{n-i}$-bounded. 
\item The objects $\unit=L_0$ and $L_{(p-2)p^{n-1}}$ are invertible.
\item All other non-zero objects in $\Ver_{p^n}$ are unbounded.
\end{enumerate}
\end{theorem}
\begin{proof}
Part (1) is a special case of part (2). To classify bounded objects we can focus on simple objects, by Theorem~\ref{ThmSimp}. Moreover, by Theorem~\ref{ThmFP}, the tensor product of two non-invertible (and non-zero) objects has at least Frobenius-Perron dimension $2$ and is therefore unbounded.

In \cite[Theorem~1.3]{BEO}, simple objects in $\Ver_{p^n}$ are classified in terms of tensor products of the simple objects $L_{ap^i}$ with $0\le i< n$ and $0\le a <p$ (and $a<p-1$ in case $i=n-1$).
By \cite[Corollary~4.44]{BEO}, we have
\begin{equation}
\label{EqFP}
\FPdim L_{ap^i}\;=\; \frac{\sin\frac{(a+1)\pi}{p^{n-i}}}{\sin\frac{\pi}{p^{n-i}}}.
\end{equation}
Consequently, as is well-known, the only invertible objects in $\Ver_{p^n}$ are $\unit=L_0$ and the odd line $\bar{\unit}=L_{(p-2)p^{n-1}}$ in $\sVec\subset \Ver_{p^n}$ in case $p>2$, proving in particular (3). 

Equation~\eqref{EqFP} further shows that $L_{p^i}$ has Frobenius-Perron dimension $2\cos(\pi/p^{n-i})$. This object, and its tensor product with $L_{(p-2)p^{n-1}}$, are therefore strictly $p^{n-i}$-bounded by Theorem~\ref{ThmFP}, proving (2). Alternatively, part~(2) also follows from the construction in \cite{BEO} and Corollary~\ref{ThmPiv2}.

To prove part (4) it now only remains for us to show that the right-hand side of \eqref{EqFP} is at least $2$ when $a>1$, ignoring the cases $a\ge p-3$, $i=n-1$. In particular, for $p=2$ there is nothing left to prove, and we assume $p>2$. We focus on the case $i<n-1$, the remaining case $i=n-1$ being similar.
Observe therefore that for $j\in\mZ_{>2}$ and $ \theta\in[0,\pi/6]$ with $j\theta\le \pi/2$, we have
$$\frac{\sin(j\theta)}{\sin(\theta)}\;\ge\; \frac{\sin(3\theta)}{\sin(\theta)}\;=\; 2\cos(2\theta)+1\;\ge\; 2.$$
We can then apply this to equation~\eqref{EqFP} to conclude the proof.
%
%
%
%
%
\end{proof}

Comparing with the classification of tensor subcategories of $\Ver_{p^n}$ in \cite[Corollary~4.61]{BEO} yields the following consequence.

\begin{corollary}
Assume $p>2$. The assignment that associates to an object the tensor subcategory it generates yields a bijection between bounded objects and tensor subcategories in $\Ver_{p^n}$.
\end{corollary}

\subsection{Special cases of Conjecture~\ref{Conjpn}}
\label{SecPfConj}

\begin{theorem}\label{Thmpn}
Conjecture~\ref{Conjpn} is valid in any symmetric tensor category which admits a symmetric tensor functor to a quasi-finite symmetric tensor category.
\end{theorem}

We start the proof with some preliminary results. Throughout we let $\cC$ be a symmetric tensor category and $X\in\cC$.

\begin{lemma}\label{LemBEO}
If $\wedge^2 X$ or $\Sym^2 X$ is invertible, then either $X$ is invertible, unbounded or strictly $p^n$-bounded for some $n\in\mZ_{>0}$.
\end{lemma}
\begin{proof}
We can easily reduce to the case where $\wedge^2 X$ is invertible, by considering $\cC\boxtimes\sVec$ when $p>2$. In that case, the result follows from \cite[Theorem~4.79]{BEO} and Theorem~\ref{ThmVer}.
\end{proof}

\begin{lemma}\label{LemEK}
If $\wedge^2 X=0$ or $\Sym^2 X=0$, then $X$ is invertible.
\end{lemma}
\begin{proof}
This follows from the proof of \cite[Lemma~5.1]{EK}.
\end{proof}

\begin{lemma}\label{LemMPI}
There are no solutions to
\begin{equation}\label{EqNoSol}
4\cos^2\left(\frac{\pi}{N}\right)\;=\; 4\cos\left(\frac{\pi}{M}\right)\,+\, 2\cos\left(j\frac{\pi}{N}\right),
\end{equation}
for $M\in\mZ_{\ge 4}$, $N\in\mZ_{>1}$ and $j\in\mZ$.
\end{lemma}
\begin{proof}
By Lemma~\ref{Lem4}(2), the algebraic number $2\cos(\pi/N)$ is totally real. Hence, for any element in the absolute Galois group $\theta\in G:=\mathrm{Gal}(\bar{\mQ}:\mQ)$, its action on left-hand side of~\eqref{EqNoSol} a positive number. For any $l\in\mZ$ with $\gcd(l,2M)=1$, there exists $\theta\in G$ with the property that
$$2\cos\left(\frac{\pi}{M}\right)=\zeta_{2M}+\zeta_{2M}^{-1}\;\,\mapsto\;\, \zeta_{2M}^l+\zeta_{2M}^{-l}=2\cos\left(l\frac{\pi}{M}\right),$$
with $\zeta_k=\exp(2\pi \iota/k)$ a primitive $k$-th root of unity. If $M$ is even, we take $l=M-1$, if $M$ is odd, we take $l=M-2$. 

We consider the case $M$ even, the case odd being almost identical. Then $\theta$ sends the right-hand side of \eqref{EqNoSol} to
$$4\cos\left(\frac{M-1}{M}\pi\right)+2\cos(\alpha)$$
for some $\alpha$. Now by assumption $M\ge 4$, hence the above is bounded from above by
$$4\cos(3\pi/4)+2\;=\;2(1-\sqrt{2})\;<\;0.$$
This shows the right-hand side is not always sent to a positive number by elements of $G$ and equation~\eqref{EqNoSol} cannot hold.
\end{proof}

\begin{proof}[Proof of Theorem~\ref{Thmpn}]
The theorem reduces, by Lemma~\ref{TensorF}, to quasi-finite tensor categories, which we focus on henceforth.
By Lemma~\ref{LemBEO} and Theorem~\ref{ThmSymm1}, it suffices to show that when $X$ is strictly $N$-bounded for $N>3$, then either $\wedge^2 X$ or $\Sym^2 X$ is invertible.

We set
$$d=\FPdim X=2\cos(\pi/N),\quad s=\FPdim \Sym^2 X,\;\mbox{and}\quad e=\FPdim \wedge^2 X.$$
In particular, $d^2=s+e$. We assume that $s\ge e$, the other case can be proved identically. It then follows that $e<2$ and, by Theorem~\ref{ThmFP}, we have $e=2\cos(\pi/M)$ for some integer $M\ge 2$. Furthermore, $s-e=2\cos(j\pi/N)$ for some $j\in\mZ$, by Example~\ref{Expsi2}. Subtracting the expressions for $s\pm e$ shows
$$4\cos\left(\frac{\pi}{M}\right)\;=\; 4\cos^2\left(\frac{\pi}{N}\right)-2\cos\left(j\frac{\pi}{N}\right).$$
By Lemma~\ref{LemMPI}, we find $M\in\{2,3\}$. If $M=2$, then  $\wedge^2X=0$, which would imply that $X$ is invertible by Lemma~\ref{LemEK}. Hence $M=3$, showing indeed that $\wedge^2 X$ is invertible.
\end{proof}

\begin{remark}\label{RemSym}
Consider $X\in \cC$ with $\wedge^2 X$ invertible. Since \cite[Theorem~4.79]{BEO} allows us to reduce to $\Ver_{p^n}$, we know that $X$ is strictly $N$-bounded for
$$N\,=\,2+\sup\{n\in\mN\,|\,\Sym^n X\not=0\}.$$ 
If $\unit\simeq \wedge^2 X$ and can be identified with $\co_{X}$, this observation also follows directly from Theorem~\ref{ThmConjP}, since then
$$\mV^n(X)\;=\;H_0(\mE^n(X))\;\simeq\;\Sym^n X.$$
\end{remark}

We conclude this section by proving Conjecture~\ref{Conjpn} without finiteness restrictions, for small $N$.

\begin{lemma}\label{Lemp}
Let $\cC$ be a symmetric tensor category over a field $\bk$.
\begin{enumerate}
\item If $X\in\cC$ is strictly $4$-bounded, then $\mathrm{char}(\bk)=2$ and $\wedge^2X=\unit$.
\item If $X\in\cC$ is strictly $5$-bounded, then $\mathrm{char}(\bk)=5$ and either $\Sym^2X$ or $\wedge^2X$ is invertible.
\item $X\in\cC$ cannot be strictly $6$-bounded.
\end{enumerate}
\end{lemma}
\begin{proof}
We will always reduce to Lemma~\ref{LemBEO}, freely use Theorem~\ref{ThmDX} and the following consequence of \eqref{DefDX} and observation (i) in the proof of Lemma~\ref{Lemgd}
\begin{equation}
\label{neweqse}
s+e\le\gd(X^{\otimes 2})=\gd(X)^2\le \bD(X)^2.
\end{equation}
Here, we abbreviated $s= \gd(\Sym^2 X)$ and $e=\gd(\wedge^2 X)$.

If $X$ is strictly $4$-bounded, then $s+e\le 2$
implies that both $\Sym^2X$ and $\wedge^2 X$ are invertible (as neither can be zero by Lemma~\ref{LemEK}).

If $X$ is strictly $5$-bounded, then
$s+e\le \bD(X)^2 \approx 2.61 $
implies that either $\Sym^2X$ or $\wedge^2 X$ is invertible, since the minimal $\gd$ of a non-invertible object is $\sqrt{2}$, by \cite[Corollary~5.2]{EK}.

Finally, assume that $X$ is strictly $6$-bounded, so $s+e\le 3$. If $p>2$ then either $\Sym^2X$ or $\wedge^2 X$ is invertible, since the minimal $\gd$ of a non-invertible object is then $(1+\sqrt{5})/2\approx 1.62$, by \cite[Corollary~5.2]{EK}.

If $p=2$, then the only other option than $\wedge^2 X$ being invertible is to have both $\wedge^2 X$ and $\Sym^2 X$ simple and hence $\Sym^2 X=\wedge^2 X$.
This means that
$$4\gd(\wedge^2 X\otimes \wedge^2 X^\ast)\;=\;\gd(X^{\otimes 2}\otimes (X^\ast)^{\otimes 2})\;=\;\bD(X)^4=9.$$
Again since the minimal $\gd$ of a non-invertible object is $\sqrt{2}$, by \cite[Corollary~5.2]{EK}, this implies that either $\wedge^2 X\otimes \wedge^2 X^\ast$ is an extension of two invertible objects, or is simple. The latter option implies that $\wedge^2 X\otimes \wedge^2 X^\ast$ is actually $\unit$, the former option that $\gd(\wedge^2 X\otimes \wedge^2 X^\ast)=2$. Both options contradict the last displayed equality.
\end{proof}

\subsection{A non-pivotal example}

In this section we construct a tensor category inspired by the representation theory of the Yangian $Y(\mathfrak{sl}_2)$, see \cite[\S 12.1]{CP}. We let $\bk$ be a field of characteristic $p>3$.

\subsubsection{} We set $\cC=\Ver_p$. We will use the interpretation of $\Ver_p$ as the semisimplification of $\Tilt SL_2$. Labelling the simple objects in $\cC$ as $L_i$ for $0\le i\le p-2$, the object $L_i$ is the image of the simple (tilting) $SL_2$-representation with highest weight $i$. For example, $L_2\in\Ver_p^+$ is the image of the adjoint representation.

 In particular, for every $X\in\cC$, we have the morphism
 $$a_X:\, L_2\otimes X\to X,$$
 which comes from the action of the Lie algebra $\mathfrak{sl}_2=\mathrm{Lie}SL_2$ on $SL_2$-representations. We also normalise the embedding
 $$D:\, L_2\xrightarrow{\sim}\wedge^2 L_2\hookrightarrow L_2^{\otimes 2}$$
 so that composition with the morphism $a_{L_2}:L_2\otimes L_2\to L_2$, coming from the bracket on~$\mathfrak{sl}_2$, is $-\id_{L_2}$.

\subsubsection{} We let $\cY$ be the abelian category of modules in $\Ver_p$ over the tensor algebra of~$L_2$. In other words, objects in $\cY$ are pairs $(X,f)$ consisting of an object $X\in\cC$ and a morphism
$$f:\, L_2\otimes X\,\to\, X.$$
Morphisms in $\cY$ are morphisms in $\cC$ leading to commutative squares.

\subsubsection{}We can define a bifunctor
$$-\otimes -\;:\; \cY\times\cY\to\cY,\quad (X,f)\otimes (Y,g)\;=\; (X\otimes Y, f\otimes Y+X\otimes g +u_{X,Y}),$$
with 
$$u_{X,Y}:=(a_X\otimes a_Y)\circ (L_2\otimes \sigma_{L_2,X}\otimes Y)\circ (D\otimes X\otimes Y).$$
Note that without the $u_{X,Y}$-term, the category $\cY$ would just be the symmetric monoidal category of modules over the (cocommutative Hopf) tensor algebra of $L_2$. The tensor product of two morphisms in $\cY$ can be computed in $\cC$; it is easily verified that the ordinary tensor product of morphisms yields again a morphism in $\cY$. 

It is also a direct exercise to verify that the associativity transformation on $\cC$ lifts to $\cY$, making $(\cY,\otimes, (\unit, 0))$ a monoidal category.

\subsubsection{}\label{Yrig}Finally, $\cY$ is rigid, and hence a tensor category with forgetful tensor functor
$$F:\cY\to\Ver_p.$$
Indeed, we can set
$$(X,f)^\ast\;:=\; (X^\ast, -(\ev_X\otimes X^\ast)\circ (X^\ast\otimes f\otimes X^\ast)\circ (\gamma_{L_2,X^\ast}\otimes \co_X)+a_{X^\ast}).$$
Indeed, a direct computation using the classical Lie algebra representation relation
\begin{equation}\label{eqCompu}(X\otimes a_{X^\ast})\circ(\gamma_{L_2, X}\otimes X^\ast)\circ (L_2\otimes\co_X)\;=\; -(a_X\otimes X^\ast)\circ (L_2\otimes \co_X)
\end{equation}
and our normalisation of $D$ show that $\co_X$ and $\ev_X$ lift to morphisms
$$\unit\to (X,f)\otimes (X,f)^\ast,\quad\mbox{and}\quad (X,f)^\ast\otimes (X,f)\to \unit,$$
and the snake relations are automatically inherited.

\subsubsection{} For any $z\in\bk$, we let $V(z)\in \cY$ denote the object $(L_1, z a_{L_1}).$
As a direct consequence of \eqref{eqCompu}, we have
$$V(z)^\ast\;\simeq\; V(z+1).$$
In particular, $V(z)$ cannot be pivotal.

\begin{prop}\label{PropYang}
Let $z\in\bk$ be arbitrary.
\begin{enumerate}
\item $V(z)\in\cY$ is $p$-bounded.
\item $V(z)\in\cY$ is  not homotopically $p$-bounded.
\end{enumerate}
\end{prop}
\begin{proof}
Part (1) is a consequence of Lemma~\ref{TensorF} and Theorem~\ref{ThmVer}(1). Part (2) follows immediately from~\autoref{LemMustPiv}.
\end{proof}

\subsection{Universal tensor categories}

\begin{question}\label{QUni}
For $N\in\mZ_{>2}\cup\{\infty\}$, is there a tensor category $\cC_N^{bnd}$ over $\bk$, so that for every tensor category $\cC$ over $\bk$ there is a bijection between isomorphism classes of strictly $N$-bounded objects $X\in\cC$ and isomorphism classes of tensor functors $\cC_N^{bnd}\to\cC$?
\end{question}

\begin{example}
We have $\cC_3^{bnd}=\Vecc_{\mZ}$, the category of finite dimensional $\mZ$-graded vector spaces.
\end{example}

\subsubsection{Discussion} There is a recipe for constructing $\cC_4^{bnd}$. Indeed, one can consider the universal pseudo-abelian $\bk$-linear rigid monoidal category generated by one object $X$, see for instance \cite{CSV}, and adjoin inverses to the morphisms in $\mE_3(X)$ and $\mE^3(X)$. The question then becomes whether the resulting category has an abelian envelope. Note that, by the theory of \cite{Sel} or \cite{BK, Kh}, there must exist a family of universal tensor tensor categories associated to this pseudo-abelian category, which together classify strictly $4$-bounded objects. The question is thus whether this family is a singleton.

For $N>4$, this recipe is no longer available, as the above only works for homotopically bounded objects. Starting at $N=5$ this is no longer sufficient, see Proposition~\ref{PropYang}.

\subsubsection{Symmetric case} Clearly, Question~\ref{QUni} has negative answer when we interpret it naively for symmetric tensor categories and functors (unless $\mathrm{char}(\bk)=2$). Indeed, for $N=3$, we have a bijection between symmetric tensor functors $\Vecc_{\mZ}\to\cC$ and invertible objects $L\in\cC$ with $\sigma_{L,L}=\id_{L^{\otimes 2}}$. For the second option, $\sigma_{L,L}=-\id_{L^{\otimes 2}}$, we consider the symmetric tensor category
$$\sVec_{\mZ}\;\subset\;\Vecc_{\mZ}\boxtimes\sVec$$
generated by the tensor product of the generators of the right-hand side. In other words, $\sVec_{\mZ}$ is the tensor category $\Vecc_{\mZ}$ equipped with symmetric braiding for the Koszul sign rule. Then symmetric tensor functors $\sVec_{\mZ}\to\cC$ classify the remaining invertible objects. 

We can extend this to higher bounded objects, at least under the assumption of Conjecture~\ref{Conjpn}.

\begin{prop}
Consider $p=\mathrm{char}(\bk)>2$, $n\in\mZ_{>0}$, and let $\cC$ be a symmetric tensor category over $\bk$.
\begin{enumerate}
\item We have a bijection between symmetric tensor functors $\Ver_{p^n}^+\boxtimes\sVec_{\mZ}\to\cC$ and strictly $p^n$-bounded $X\in\cC$ with $\wedge^2 X$ invertible.
\item We have a bijection between symmetric tensor functors $\Ver_{p^n}^+\boxtimes\Vecc_{\mZ}\to\cC$ and strictly $p^n$-bounded $X\in\cC$ with $\Sym^2 X$ invertible.
\end{enumerate}
\end{prop} 
\begin{proof}
A classification of objects with $\wedge^2 X$ invertible is given in \cite[Theorem~4.79]{BEO}. In particular it follows that the categories of the form ``$\cC'_{\mZ}$'' are the universal symmetric tensor category generated by such an object. For $p^n$, this category $\cC'_{\mZ}$ is defined as the tensor subcategory of
$$\Ver_{p^n}\boxtimes \Vecc_{\mZ}$$
generated by $L_1\boxtimes g$ with $g$ a one-dimensional vector space in degree 1. Using 
$$\Ver_{p^n}\;\simeq\; \Ver_{p^n}^+\boxtimes \sVec, \quad L_1\mapsto (L_1\otimes\bar{\unit})\boxtimes \bar{\unit}$$
shows that $\cC'_{\mZ}$ is $\Ver_{p^n}^+\boxtimes\sVec_{\mZ}$, proving part (1). Note that we write $\bar{\unit}$ for the odd line in $\sVec\subset\Ver_{p^n}$.

We can reduce part (2) to part (1). First assume that we have $X\in\cC$ with $\Sym^2 X$ invertible. Via part (1) we obtain a symmetric tensor functor
$$\Ver^+_{p^n}\boxtimes\sVec_{\mZ}\;\to\; \cC\boxtimes\sVec,\quad (L_1\otimes\bar{\unit})\boxtimes g\mapsto X\boxtimes\bar{\unit}.$$
Taking the tensor product on both sides with $\sVec$ and using two obvious equivalences yields a symmetric tensor functor
$$\Ver^+_{p^n}\boxtimes\Vecc_{\mZ}\boxtimes\sVec\;\to\; \cC\boxtimes\Vecc_{\mZ/2}\boxtimes\sVec,\quad (L_1\otimes\bar{\unit})\boxtimes g\boxtimes \unit\mapsto X\boxtimes g\boxtimes \unit,$$
where we keep using $g$ for generators of cyclic pointed categories.
Of course, now the displayed action no longer determines the functor, but it shows that we can restrict the functor to subcategories, and compose with the functor forgetting grading,
$$\Ver^+_{p^n}\boxtimes\Vecc_{\mZ}\;\to\; \cC\boxtimes\Vecc_{\mZ/2}\;\to\;\cC,\quad (L_1\otimes\bar{\unit})\boxtimes g\mapsto X\boxtimes g\mapsto X.$$
Hence, we find a suitable functor. To show that an object leads to a unique functor, we can again reduce to part (1). Indeed, for two functors from $\Ver_{p^n}^+\boxtimes\Vecc_{\mZ}$ to $\cC$ we can consider two functors from $\Ver_{p^n}^+\boxtimes\sVec_{\mZ}$ to $\cC\boxtimes\sVec$.
\end{proof}

\begin{remark}
Similarly, for $p=2$ and $n\in\mZ_{>1}$, strictly $2^n$-bounded objects $X$ (with $\wedge^2 X$ invertible, as is always satisfied conjecturally) correspond to symmetric tensor functors out of the tensor subcategory of
$\Ver_{2^n}\boxtimes \Vecc_{\mZ}$
generated by the tensor product of the generators.
\end{remark}

\section{Finite symmetric and exterior algebras}\label{SecFiniteP}

Let $\cC$ be a symmetric tensor category over an algebraically closed field $\bk$ with $p=\mathrm{char}(\bk)>2$.

\subsection{Maximal degrees}

\subsubsection{} For $X\in\cC$, we set
$$m(X):=\sup\{m\in\mN\mid \Sym^m X\not=0\}\;\in\;\mN\cup\{\infty\},$$
and similarly
$$n(X):=\sup\{n\in\mN\mid \wedge^n(X)\not=0\}\;\in\;\mN\cup\{\infty\}.$$
We can characterise $m(X)$ alternatively.

\begin{lemma}\label{LemSymInv}
Take $0\not=X\in\cC$ which is not invertible. There is at most one $m\in\mZ_{>0}$ for which $\Sym^mX$ is invertible. If such $m$ exists, then $m=m(X)$; if such $m$ does not exist, then $m(X)=\infty$.
\end{lemma}
\begin{proof}
As is shown in \cite[Lemma~7.4.5]{CEO3}, for any $X$, if $m\in\mN$ is the maximal degree for which $\Sym^m X\not=0$, then $\Sym^mX$ is invertible.

To conclude the proof we thus need to show the converse claim: If $X\not=0$ is not invertible, but $\Sym^mX$ is invertible for $m\in\mZ_{>0}$, then $\Sym^{m+1}X=0$.
If $X$ is simple (and non-invertible) this claim is proved in \cite[Corollary~7.4.4]{CEO3}. Consider therefore a short exact sequence in $\cC$
$$0\to Y\to X\to Z\to 0$$
with both $Y$ and $Z$ non-zero. If we consider the filtration on $\Sym^d X$ induced by the short exact sequence, then \cite[Theorem~3.2.2]{CEO3} implies that
\begin{eqnarray*}
\gr \Sym^d X&\simeq& \bigoplus_{i=0}^d(\Sym^i Y)/I^i\otimes \Sym^{d-i}Z\\
&\simeq&\Sym^dZ\oplus(Y\otimes\Sym^{d-1}Z)\oplus\cdots \oplus(Z\otimes (\Sym^{d-1} Y)/I^{d-1})\oplus (\Sym^d Y)/I^d,
\end{eqnarray*}
where $I=\oplus_iI^i$ is the kernel of the graded (by degree) algebra morphism $\Sym Y\to\Sym X$.
If $\Sym^d X$ is invertible for $d>1$, then analysing the first two terms shows that $\Sym^dZ=0$, so there exists $1\le m< d$ for which $\Sym^i Z\not=0$ for $i\le m$ and $\Sym^iZ=0$ for $i> m$. Invertibility of $\Sym^dX$ then also implies that $I^{d-m+1}=\Sym^{d-m+1}Y$, showing that $\Sym^{d+1}X=0$.
\end{proof}

\begin{remark}
The relation in Lemma~\ref{LemSymInv} for symmetric powers also exists between the invertibility and vanishing of $\wedge^d X$. This can either be proved with the same approach, but also follows as a consequence of Lemma~\ref{LemSymInv}, considering $X^\ast\boxtimes\bar{\unit}$ in $\cC\boxtimes\sVec$.
\end{remark}

\subsection{Connection with boundedness}

\begin{example}Consider $X\in\cC$.
\begin{enumerate}
\item $X=0$ if and only if $X$ is 2-bounded if and only if $(m(X),n(X))=(0,0)$.
\item $X$ is invertible if and only if $X$ is 3-bounded if and only if $(m(X),n(X))$ is $(\infty,1)$ or $(1,\infty)$.
\end{enumerate}
\end{example}

At least conjecturally, all boundedness can be described in terms of $m(X)$ and $n(X)$.

\begin{prop}
Let $\cC$ be as in Theorem~\ref{Thmpn}. Then $X\in \cC$ is strictly $N$-bounded, for $N>3$ if and only if $(m(X),n(X))$ is $(N-2,2)$ or $(2,N-2)$.
\end{prop}
\begin{proof}
This follows from Theorem~\ref{Thmpn} and Remark~\ref{RemSym}.
\end{proof}

\subsubsection{} Motivated by the proposition, for $X\in\cC$ we set
$$\mathbf{N}(X):=m(X)+n(X)\in\mN\cup\{\infty\}.$$
If $X$ is strictly $N$-bounded, for $3<N$, then in particular $\mathbf{N}(X)=N$.

\begin{remark}
We could have replaced symmetric powers by divided powers and/or replaced $\wedge^mX$ by $\wedge^m(X^\ast)^\ast$ in the definitions of $m(X)$ and $n(X)$. In general, this would lead to different numbers, see \cite[\S 10]{CEO3}.
\end{remark}

\begin{question}\label{QueKoszul}
Consider $X\in\cC$ with finite $m=m(X)$ and $n=n(X)$. Is the symmetric algebra $\Sym X$ an $(m,n)$-Koszul algebra, according to \cite[Definition~5.3]{Et}?
\end{question}
The case of the bounded objects has an affirmative answer: If $m=m(X)<\infty$ and $n(X)=2$, then $\Sym X$ is $(m,2)$-Koszul, by \cite[Corollary~6.6(ii)]{BE} and \cite[Proposition~4.79]{BEO}. The case $m(X)=2$ then follows from \cite[Proposition~5.4]{Et}. 

\begin{lemma}\label{LemDim}
For every $0\not=X$ we have $\mathbf{N}(X)\ge p$. In fact, $\mathbf{N}(X)$ is a multiple of $p$.
\end{lemma}
\begin{proof}
We use the $p$-adic dimensions $\Dim_{\pm}\in\mZ_p$ from \cite{EHO}. If $\mathbf{N}(X)<\infty$, then
$$\Dim_+(X)=-m(X)\quad\mbox{and}\quad\Dim_-(X^\ast)=n(X).$$
Both these dimensions should reduce to the same element $\dim(X)=\dim(X^\ast)$ of the prime field $\mF_p\subset\bk$ under $\mZ\subset\mZ_p\tto\mF_p$. Hence, their difference $m(X)+n(X)$ is divisible by $p$.
\end{proof}

\subsection*{Acknowledgement}

The authors thank Mikhail Kapranov for sharing and explaining the work in \cite{DKS}. The first author thanks Hohto Bekki for a discussion leading to the current proof of Lemma~\ref{LemMPI} and MPIM Bonn for excellent working conditions.

K. C.'s work was partially supported by the ARC grants DP210100251 and FT220100125. P. E.'s work was partially supported by the NSF grant DMS - 2001318.

\end{document}